\documentclass[11pt]{amsart}

\usepackage{amsmath,amsthm,amssymb,verbatim,amscd,easybmat, mathrsfs}
\usepackage[tmargin=1.0in,bmargin=1.0in,rmargin=1.0in,lmargin=1.0in]{geometry}
\usepackage[breaklinks=true]{hyperref}

\theoremstyle{plain}

\newtheorem{theorem}{\textrm{\textbf{Theorem}}}[section]

\newtheorem{utheorem}{\textrm{\textbf{Theorem}}}

\newtheorem{corollary}[theorem]{\textrm{\textbf{Corollary}}}
\newtheorem{proposition}[theorem]{\textrm{\textbf{Proposition}}}
\newtheorem{lemma}[theorem]{\textrm{\textbf{Lemma}}}

\theoremstyle{definition}

\newtheorem{definition}[theorem]{\textrm{\textbf{Definition}}}

\newtheorem{remark}[theorem]{\textrm{\textbf{Remark}}}

\theoremstyle{remark}

\numberwithin{equation}{section}

\def\rk {\mathop{\rm rank}}

\def\R{\mathbb{R}}
\def\C{\mathbb{C}}

\def\N{\mathbb{N}}

\def\br{\mathbb{S}}
\newcommand{\bp}{\ensuremath{\mathbb P}}
\def\br{\mathbb{S}}
\def\one#1{{\bf 1}_{#1}}

\def\mc#1{\multicolumn{1}{|c}{#1}}

\DeclareMathOperator{\Id}{\ensuremath{Id}}

\begin{document}
\title{Preserving positivity for matrices with sparsity constraints}

\author{Dominique Guillot \and Apoorva Khare \and Bala Rajaratnam}

\thanks{D.G., A.K., and B.R.~are partially supported by the following: US
Air Force Office of Scientific Research grant award FA9550-13-1-0043, US
National Science Foundation under grant DMS-0906392, DMS-CMG 1025465,
AGS-1003823, DMS-1106642, DMS-CAREER-1352656, Defense Advanced Research
Projects Agency DARPA YFA N66001-11-1-4131, the UPS Foundation, SMC-DBNKY,
and an NSERC postdoctoral fellowship}

\address[D.G.]{Department of Mathematical Sciences, University of
Delaware, Newark, DE - 19716, USA}

\address[A.K. and B.R.]{Departments of Mathematics and Statistics,
Stanford University, Stanford, CA - 94305, USA}

\email{D.G. dguillot@udel.edu; A.K. khare@stanford.edu; B.R.
brajaratnam01@gmail.com}

\date{\today}



\begin{abstract}
Functions preserving Loewner positivity when applied entrywise to
positive semidefinite matrices have been widely studied in the
literature. Following the work of Schoenberg [Duke Math.~J.~9], Rudin
[Duke Math.~J.~26], and others, it is well-known that functions
preserving positivity for matrices of all dimensions are absolutely
monotonic (i.e., analytic with nonnegative Taylor coefficients). In this
paper, we study functions preserving positivity when applied entrywise to
sparse matrices, with zeros encoded by a graph $G$ or a family of graphs
$G_n$.
Our results generalize Schoenberg and Rudin's results to a modern
setting, where functions are often applied entrywise to sparse matrices
in order to improve their properties (e.g.~better conditioning). The only
such result known in the literature is for the complete graph $K_2$. We
provide the first such characterization result for a large family of
non-complete graphs. Specifically, we characterize functions preserving
Loewner positivity on matrices with zeros according to a tree.
These functions are multiplicatively midpoint-convex and super-additive.
Leveraging the underlying sparsity in matrices thus admits the use of
functions which are not necessarily analytic nor absolutely monotonic. We
further show that analytic functions preserving positivity on matrices
with zeros according to trees can contain arbitrarily long sequences of
negative coefficients, thus obviating the need for absolute monotonicity
in a very strong sense. This result leads to the question of exactly when
absolute monotonicity is necessary when preserving positivity for an
arbitrary class of graphs. 
We then provide a stronger condition in terms of the numerical range of
all symmetric matrices, such that functions satisfying this condition on
matrices with zeros according to any family of graphs with unbounded
degrees are necessarily absolutely monotonic.
\end{abstract}
\maketitle

\section{Introduction and main results}

Functions preserving Loewner positivity  when applied entrywise to
positive semidefinite matrices have been well-studied in the literature
(see e.g. Schoenberg \cite{Schoenberg42}, Rudin \cite{Rudin59}, Herz
\cite{Herz63}, Horn \cite{Horn}, Christensen and Ressel
\cite{Christensen_et_al78}, Vasudeva \cite{vasudeva79}, FitzGerald et al
\cite{fitzgerald}). An important characterization of functions $f:(-1,1)
\rightarrow \mathbb{R}$ such that $f[A] := (f(a_{ij}))$ is positive
semidefinite for all positive semidefinite matrix $A = (a_{ij})$ of all
dimensions $n$ with entries in $(-1,1)$ has been obtained by Schoenberg
and Rudin (\cite{Schoenberg42}, \cite{Rudin59}). Their results show that
such functions are absolutely monotonic (i.e., analytic with nonnegative
Taylor coefficients).

In modern applications, functions are often applied entrywise to positive
semidefinite matrices (e.g.~covariance/correlation matrices) in order to
improve their properties such as better conditioning or to induce a
Markov random field structure (see
\cite{Guillot_Rajaratnam2012,Guillot_Rajaratnam2012b}). Understanding if
and how positivity is preserved is critical for these procedures to be
widely applicable. In such settings, various distinguished submanifolds
of the cone of positive semidefinite matrices are of particular interest.
Two important cases naturally arising in modern applications involve (1)
constraining the rank, and (2) constraining the sparsity of correlation
matrices. The rank of a sample correlation matrix corresponds to the
sample size of the population used to estimate it. It is thus natural to
ask which functions preserve Loewner positivity when applied entrywise to
positive semidefinite matrices of a given rank. This analysis was carried
out in \cite{GKR-lowrank}. There it was shown that functions preserving
positivity when applied entrywise to matrices of rank $1$ or $2$ are
automatically absolutely monotonic. Thus, preserving positivity for small
subsets of the cone of positive semidefinite matrices immediately forces
the function to be absolutely monotonic. The converse of this result
(i.e., that every absolutely monotonic function preserves Loewner
positivity) follows immediately from the Schur product theorem. 

In this paper, we study the second important problem: preserving
positivity when sparsity constraints are imposed. The sparsity pattern of
a matrix $A = (a_{ij})$ is naturally encoded by a graph $G = (V,E)$ where
$V = \{1,\dots,n\}$ and $(i,j) \not\in E$ if $a_{ij} = 0$. Thus, our goal
is to study functions preserving positivity when applied entrywise to
positive semidefinite matrices with zeros according to a fixed graph $G$,
or a family of graphs $(G_n)_{n \geq 1}$. In particular, when $G_n = K_n$
(the complete graph on $n$ vertices) for all $n$, the problem reduces to
the classical problem studied by Schoenberg, Rudin, and others.

Positive semidefinite matrices with zeros according to graphs arise
naturally in many applications. For example, in the theory of Markov
random fields in probability theory (\cite{lauritzen,whittaker}), the
nodes of a graph $G$ represent components of a random vector, and edges
represent the dependency structure between nodes. Thus, absence of an
edge implies marginal or conditional independence between the
corresponding random variables, and leads to zeros in the associated
covariance or correlation matrix (or its inverse). Such models therefore
yield parsimonious representations of dependency structures.
Characterizing entrywise functions preserving Loewner positivity for
matrices with zeros according to a graph is thus of tremendous interest
for modern applications. Obtaining such characterizations is, however,
much more involved than the original problem considered by Schoenberg and
Rudin, as one has to enforce and maintain the sparsity constraint. The
problem of characterizing functions preserving positivity for sparse
matrices is also intimately linked to problems in spectral graph theory
and many other problems (see e.g.~\cite{Agler_et_al_88, Brualdi_mutually,
Hogben_2005, pinkus04}).

We now state the main results in this paper. To do so, we first introduce
some notation. Let $G = (V,E)$ be a graph with vertex set $V =
\{1,\dots,n\}$. Denote by $|G| := |V|$ and by $\Delta(G)$ the maximum
degree of the vertices of $G$. Given a subset $I \subset \mathbb{R}$, let
$\br_n(I)$ denote the space of $n \times n$ symmetric matrices with
entries in $I$, and $\bp_n(I)$ be the cone of real $n \times n$ positive
semidefinite matrices with entries in $I$. Define $\br_G(I)$ and
$\bp_G(I)$ to be the respective subsets of matrices with zeros according
to $G$:
\begin{equation}\label{eqn:def_S_G}
\br_G(I) := \{A \in \br_{|G|}(I) : a_{ij} = 0 \textrm{ for every } (i,j)
\not\in E, i \not=j\}, \quad \bp_G(I) := \bp_{|G|}(I) \cap \br_G(I).
\end{equation}

\noindent We denote $\br_n(\R)$ and $\bp_n(\R)$ respectively by $\br_n$
and $\bp_n$ for convenience. Given a function $f: \mathbb{R} \rightarrow
\mathbb{R}$ and $A \in \br_{|G|}(\R)$, denote by $f_G[A]$ the matrix
\begin{equation}
(f_G[A])_{ij} := \begin{cases}
f(a_{ij}) & \textrm{ if } (i,j) \in E \textrm{ or } i = j, \\
0 & \textrm{otherwise.}
\end{cases}
\end{equation} 

In the case where $G =K_n$, the complete graph on $n$ vertices, we denote
$f_{K_n}[A]$ by $f[A]$. Schoenberg and Rudin's result can now be rephrased
by saying that $f_{K_n}[A] \in \bp_{K_n}(\mathbb{R})$ for all $n \geq 1$
and all $A \in \bp_{K_n}(-1,1)$ if and only if $f$ has a power series
representation with nonnegative coefficients.

In this paper, we generalize Schoenberg and Rudin's result by considering
functions $f$ mapping $\bp_G$ into itself for other important families of
graphs. As we show, this problem is much more involved for non-complete
graphs than the special case considered by Schoenberg, Rudin, and others.
In fact, such characterization results are only known for (a) the family
of all complete graphs $K_n$ - by the work of Schoenberg and Rudin; see
Theorem \ref{Therz}; and (b) the single graph $K_2$ - by the work of
Vasudeva \cite[Theorem 2]{vasudeva79} - see Theorem
\ref{thm:vasudeva:M2}. However, to our knowledge, no other
characterization result has been proved since Vasudeva's work in 1979 for
$K_2$. Our first main result in this paper is a characterization result
for all trees.

\begin{utheorem}\label{thmA}
Suppose $I = [0,R)$ for some $0 < R \leq \infty$, and $f : I
\to [0,\infty)$. Let $G$ be a tree with at least $3$ vertices, and let
$A_3$ denote the path graph on $3$ vertices. Then the following are
equivalent:
\begin{enumerate}
\item $f_G[A] \in \bp_G$ for every $A \in \bp_G(I)$;
\item $f_T[A] \in \bp_T$ for all trees $T$ and all matrices $A \in
\bp_T(I)$;
\item $f_{A_3}[A] \in \bp_{A_3}$ for every $A \in \bp_{A_3}(I)$; 
\item The function $f$ satisfies:
\begin{equation}\label{Emidconvex}
f(\sqrt{xy})^2 \leq f(x) f(y), \qquad \forall x,y \in I
\end{equation}
and is superadditive on $I$, i.e., 
\begin{equation}\label{Esuper}
f(x+y) \geq f(x) + f(y), \qquad \forall x,y,x+y \in I.
\end{equation}
\end{enumerate}
\end{utheorem}

\noindent Note that some sources refer to \eqref{Emidconvex} as {\it
mid(point)-convexity} for the function $x \mapsto \log f(e^x)$, albeit on
an interval different from $(0,R)$. Thus functions preserving positivity
for trees coincide with the class of midpoint convex superadditive
functions. 

Recall that previous results by Schoenberg and Rudin show that entrywise
functions preserving positivity for all matrices (i.e., according to the
family of complete graphs $K_n$ for $n \geq 1$) are absolutely monotonic
on the positive axis. It is not clear if functions satisfying
\eqref{Emidconvex} and \eqref{Esuper} in Theorem \ref{thmA} are
necessarily absolutely monotonic, or even analytic. We show below in
Proposition \ref{thm:frac_tree} that such functions need not be analytic.
Our second main result demonstrates that even if the function is
analytic, it can in fact have arbitrarily long strings of negative Taylor
coefficients.

\begin{utheorem}\label{thmB}
There exists a function $f(z) = \sum_{n=0}^\infty a_n z^n$ analytic on
$\mathbb{C}$ such that 
\begin{enumerate}
\item $a_n \in [-1,1]$ for every $n \geq 0$; 
\item The sequence $(a_n)_{n \geq 0}$ contains arbitrarily long strings
of negative numbers; 
\item For every tree $G$, $f_G[A] \in \bp_G$ for every $A \in
\bp_G([0,\infty))$.
\end{enumerate}
In particular, if $\Delta(G)$ denotes the maximum degree of the vertices
of $G$, then there exists a family $G_n$ of graphs and an analytic
function $f$ that is not absolutely monotonic, such that:
\begin{enumerate}
\item $\sup_{n \geq 1} \Delta(G_n) = \infty$; 
\item $f_{G_n}[A] \in \bp_{G_n}$ for every $A \in \bp_{G_n}([0,\infty))$.
\end{enumerate}
\end{utheorem}

\noindent As we will show, it is even possible to choose $f$ to be a real
polynomial of degree $n \geq 4$ preserving $\bp_G$ for all trees $G$, and
with up to $n-3$ negative coefficients.

Theorem \ref{thmB} demonstrates that functions preserving positivity for
a general family of graphs $G_n$ with unbounded degree are not
necessarily absolutely monotonic. It is natural to seek minimal
additional restrictions on a family of graphs $\{G_n\}_{n \geq 1}$ and a
function $f$ mapping $\bp_{G_n}$ into itself for all $n \geq 1$, in order
to conclude that $f$ is analytic and absolutely monotonic on
$[0,\infty)$. Our last main result provides such a sufficient condition.

\begin{utheorem}\label{thmC}
Let $\{G_n\}_{n \geq 1}$ be a family of graphs such that
\[ \sup_{n \geq 1} \Delta(G_n) = \infty. \]
Let $I := [0,R)$ for some $0 < R \leq \infty$ and let $f: I \rightarrow
\R$ be a function such that for every $n \geq 1$, $\beta^T f_{G_n}[M]
\beta \geq 0$ for every symmetric matrix $M \in \br_{G_n}(I)$, and every
$\beta \in \R^{|G_n|}$ such that $\beta^T M \beta \geq 0$. Then $f$ is
analytic and absolutely monotonic on $I$.
\end{utheorem}

In other words, if one wants to preserve a weaker form of positivity as
given in Theorem \ref{thmC} and simultaneously to be able to use
functions that are not absolutely monotonic, then the sequence of graphs
$\{ G_n \}_{n \geq 1}$ has to be of bounded degree. Thus this notion of
preserving positivity necessitates a specific form of sparsity in terms
of the degrees of the associated nodes.

\begin{remark}
Recall that the numerical range of a $n \times n$ matrix $A$ is given by
\begin{equation*}
W(A) := \{\beta^* A \beta : \beta \in \C^n, \beta^* \beta = 1\}, 
\end{equation*}
where $\beta^*$ denotes the conjugate transpose of $\beta$. When $A$ is
Hermitian,  it is clear that $W(A) \subset \R$. Moreover, $A$ is positive
semidefinite if and only if $W(A) \subset [0,\infty)$, i.e., $W(A) =
W(A)_+$ where $W(A)_+ := W(A) \cap [0,\infty)$. Thus $f$ preserves
positivity on $\bp_n(\R)$ if and only if $W(f[A]) \subset [0,\infty)$ for
all matrices $A \in \br_n(\R)$ such that $W(A) = W(A)_+$. In Theorem
\ref{thmC}, this condition is strengthened in the hypothesis by
considering the effect of $f$ on the positive part of the numerical range
of all matrices $A \in \br_n(\R)$.
\end{remark}

The remainder of the paper is structured as follows. Section \ref{S2}
reviews many important characterizations of functions preserving
positivity in various settings. In Section \ref{S3}, we study the
properties of positive semidefinite matrices with zeros according to a
tree, and prove Theorem \ref{thmA}. As an application of Theorem
\ref{thmA}, in Section \ref{S4}, we show that $x \mapsto x^\alpha$
preserves $\bp_G$ for any tree $G$ if and only if $\alpha \geq 1$. Thus
the phase transition, or {\it critical exponent} for preserving
positivity on $\bp_G$ occurs at $\alpha = 1$ (see
e.g.~\cite{Bhatia-Elsner, FitzHorn, Guillot_Khare_Rajaratnam-CEC,
GKR-complex, GKR-crit-2sided, Hiai2009} for more details about critical
exponents).
We then prove Theorem \ref{thmB} by showing that there exist polynomials
and more general analytic functions with large numbers of negative
coefficients, which preserve $\bp_G$ for every tree $G$. This provides a
negative answer to a natural generalization of Schoenberg and Rudin's
results when the problem of preserving positivity is restricted to sparse
positive semidefinite matrices. Finally in Section \ref{S5}, we present
natural stronger conditions for preserving positivity, such that the
functions satisfying them are necessarily absolutely monotonic.\bigskip

\noindent \emph{Notation:}
In this paper, all graphs $G = (V,E)$ are finite, undirected, with no
self-loops. We denote by $|G|$ the cardinality of $V$. We let $K_n$ and
$A_n$ denote the complete graph and the path graph on $n$ vertices
respectively. The $n \times n$ identity matrix is denoted by $\Id_n$. We
denote by ${\bf 0}_{m \times n}$ and ${\bf 1}_{m \times n}$ the $m \times
n$ matrices with all entries equal to $0$ and $1$ respectively. 

\section{Literature review}\label{S2}

Characterizing functions which preserve some form of positivity of
matrices has been studied by many authors in the literature including
Schoenberg, Rudin, Herz, Horn, Vasudeva, Christensen and Ressel,
FitzGerald, Micchelli, and Pinkus, and more recently, Hansen, Hiai,
Bharali and Holtz, as well as the authors. The notion of
absolute monotonicity is crucial in many of these results. We begin by
reviewing important properties of these functions. 

\begin{definition}
Let $I \subset \mathbb{R}$ be an interval with interior $I^\circ$. A
function $f \in C(I)$ is said to be \emph{absolutely monotonic} on $I$ if
it is in $C^\infty(I^\circ)$ and $f^{(k)}(x) \geq 0$ for every $x \in
I^\circ$ and every $k \geq 0$. 
\end{definition}

It is not immediate that if $f$ is absolutely monotonic on $[0,\infty)$,
then $f$ is entire - however, the following result shows that this is
indeed true. Recall that the $n$-th forward difference of a function $f$,
with step $h>0$ at the point $x$, is given by
\[ \Delta^n_h[f](x) := \sum_{i=0}^n (-1)^i \binom{n}{i} f(x+(n-i)h). \]

\begin{theorem}[see {\cite[Chapter IV, Theorem
7]{widder}}]\label{thm:abs_monotonic_equiv}
Let $0 < R \leq \infty$ and let $f: [0,R) \rightarrow \R$. Then the
following are equivalent:
\begin{enumerate}
\item $f$ is absolutely monotonic on $[0,R)$.

\item $f$ can be extended analytically to the complex disc $D(0,R) := \{z
\in \mathbb{C} : |z| < R\}$, and 
$f(z) = \sum_{n=0}^\infty a_n z^n$ on $D(0,R)$, 
for some $a_n \geq 0$.

\item For every $n \geq 1$, $\Delta^n_h[f](x) \geq 0$
for all non-negative integers $n$ and for all $x$ and $h$ such that 
$0 \leq x < x+h < \dots < x+nh < R$.
\end{enumerate}
\end{theorem}


One of the main results in the literature on preserving positive
semidefiniteness was proved under various restrictions by multiple
authors. We only write down the most general version here. 

\begin{theorem}[see Schoenberg \cite{Schoenberg42}, Rudin \cite{Rudin59},
Vasudeva \cite{vasudeva79},  Herz \cite{Herz63}, Horn \cite{Horn},
Christensen and Ressel \cite{Christensen_et_al78}, FitzGerald et al.
\cite{fitzgerald}, Hiai \cite{Hiai2009}]\label{Therz}
Suppose $0 < R \leq \infty$, and $f : (-R,R) \to \R$.
Set $I := (-R, R)$. Then the following are equivalent:
\begin{enumerate}
\item For all $n \geq 1$ and $A \in \bp_n(I)$, $f[A] \in \bp_n$.
\item $f$ is analytic on the complex disc $D(0,R)$ and absolutely
monotonic on $(0,R)$. Equivalently, $f$ admits a power series
representation
$f(x) = \sum_{n=0}^\infty a_n x^n$ on $(-R,R)$
for some coefficients $a_n \geq 0$.
\end{enumerate}
\end{theorem}

The statement of Theorem \ref{Therz} for $R = \infty$ is very similar to
earlier results by Vasudeva \cite{vasudeva79}, which were extended in
previous work \cite{GKR-lowrank}. Once again we write down the most
general version here.

\begin{theorem}[Vasudeva \cite{vasudeva79}; Guillot, Khare, and
Rajaratnam \cite{GKR-lowrank}]\label{Tvasu}
Let $0 \leq a < b \leq \infty$. Assume $I = (a,b)$ or $I =
[a,b)$ and let $f: I \rightarrow \mathbb{R}$. Then each of the
following assertions implies the next: 
\begin{enumerate}
\item The function $f$ can be extended analytically to $D(0,b)$ and $f(z)
= \sum_{n=0}^\infty c_n z^n$ on $D(0,b)$, for some $c_n \geq 0$; 
\item For all $n \geq 1$ and $A \in \bp_n(I)$, $f[A] \in \bp_n$;
\item $f$ is absolutely monotonic on $I$.
\end{enumerate}  
If furthermore, $0 \in I$, then $(3) \Rightarrow (1)$ and so all the
assertions are equivalent. 
\end{theorem}

Note that in all the previous results, the dimension $n$ is allowed to
grow to infinity. When the dimension is fixed, the problem is much more
involved and very few results are known. The following necessary
condition was shown by Horn \cite{Horn} (and attributed to Loewner).

\begin{theorem}[Horn \cite{Horn}]\label{Thorn}
Suppose $f : (0,\infty) \to \R$ is continuous. Fix $2 \leq n \in \N$ and
suppose that $f[A] \in \bp_n$ for all $A \in \bp_n((0,\infty))$. Then $f
\in C^{n-3}((0,\infty))$,
\[ f^{(k)}(x) \geq 0, \qquad \forall x > 0,\ 0 \leq k \leq n-3, \]

\noindent and $f^{(n-3)}$ is a convex non-decreasing function on
$(0,\infty)$. In particular, if $f \in C^{n-1}((0,\infty))$, then
$f^{(k)}(x) \geq 0$ for all $x > 0$ and $0 \leq k \leq n-1$.
\end{theorem}

Note that preserving positivity on only a small subset of the matrices in
$\bp_n$ (for fixed $n$) guarantees that $f$ is highly differentiable on
$I$ with nonnegative derivatives. Moreover, applying Theorem \ref{Thorn}
for all $n \in \N$ easily yields Theorem \ref{Tvasu} for $I = (0,
\infty)$ as a special case. When $n=2$, the following characterization of
entrywise functions preserving positivity on $\bp_2((0,\infty))$ was
shown by Vasudeva \cite[Theorem 2]{vasudeva79}. To the authors'
knowledge, no characterization is known when $n>2$.

\begin{theorem}[Vasudeva \cite{vasudeva79}; Guillot, Khare, and
Rajaratnam \cite{GKR-lowrank}]\label{thm:vasudeva:M2}
Let $I \subset \mathbb{R}$ be an interval such that $|\inf I| \leq \sup I
> 0$, $I \cap (0,\infty)$ is open, and let $f: I \rightarrow \mathbb{R}$.
Then the following are equivalent:
\begin{enumerate}
\item $f[A] \in \bp_2$ for every $2 \times 2$ matrix $A
\in \bp_2(I)$.
\item $f$ satisfies: $f(\sqrt{xy})^2 \leq f(x) f(y)$ for all $x,y \in I
\cap [0, \infty)$, and $|f(x)| \leq f(y)$ whenever $|x| \leq y \in I$.
\end{enumerate}

\noindent In particular, if $(1)$ holds, then either $f \equiv 0$ on $I
\setminus \{ \pm \sup I \}$, or $f(x) > 0$ for all $x \in I \cap (0,
\infty)$. Moreover $f$ is continuous on $(0,\infty) \cap I$.
\end{theorem}

\begin{remark}\label{rem:vasudevaM2}
If $G$ is a graph with at least one edge and $f_G[-]$ preserves
$\bp_G([0,R))$, then $f_{K_2}[-]$ preserves $\bp_2([0,R))$ by considering
matrices of the form $A \oplus \Id_{n-2}$. Hence all of the assertions in
Theorem \ref{thm:vasudeva:M2} hold when $I = [0,R)$ and $G$ is nonempty. 
\end{remark}

Recall that in applications, functions are often applied entrywise to
covariance/correlation matrices to improve properties such as their
condition number (see
e.g.~\cite{Guillot_Rajaratnam2012,Guillot_Rajaratnam2012b}). In that
setting, the rank of a sample correlation matrix corresponds to the
sample size of the population used to estimate the matrix. With this
application in mind, the following characterization in fixed dimension
was obtained in \cite{GKR-lowrank} under additional rank constraints.
Define $\br_n^k(I) := \{A \in \br_n(I) : \rk A \leq k\}$ and $\bp_n^k(I)
:= \{A \in \bp_n(I) : \rk A \leq k\}$.

\begin{theorem}[Guillot, Khare, and Rajaratnam, {\cite[Theorem
B]{GKR-lowrank}}]\label{Tlowrank}
Let $0 < R \leq \infty$ and $I = [0, R)$ or $(-R,R)$. Fix integers $n
\geq 2$, $1 \leq k < n-1$, and $2 \leq l \leq n$. Suppose $f \in C^k(I)$.
Then the following are equivalent: 
\begin{enumerate}
\item $f[A] \in \br_n^k$ for all $A \in \bp_n^l(I)$; 
\item $f(x) = \sum_{t=1}^r a_t x^{i_t}$ for some $a_t \in \R$ and some
$i_t \in \N$ such that 
\begin{equation}\label{eqn:sum_binom}
\sum_{t=1}^r \binom{i_t+l-1}{l-1} \leq k. 
\end{equation}
\end{enumerate}

\noindent Similarly, $f[-]: \bp_n^l(I) \to \bp_n^k$ if and only if $f$
satisfies (2) and $a_i \geq 0$ for all $i$. Moreover, if $I = [0,R)$ and
$k \leq n-3$, then the assumption that $f \in C^k(I)$ is not required.
\end{theorem}

Many other interesting characterizations have also been obtained in other
settings. In \cite{bharali}, Bharali and Holtz characterize entire functions $f$ such that $f(A)$ is
entrywise nonnegative for every entrywise nonnegative and triangular
matrix $A$ (here $f(A)$ is computed using the functional calculus). In \cite{Guillot_Rajaratnam2012b}, Guillot and Rajaratnam
generalize the classical results of Schoenberg and Rudin to the case
where the function is only applied to the off-diagonal elements of
matrices (as is often the case in applications when regularizing positive semidefinite matrices).
Hansen \cite{Hansen92} and Micchelli and Willoughby
\cite{micchelli_et_al_1979} also characterize functions preserving
entrywise nonnegativity when applied to symmetric matrices using the
functional calculus.

\section{Characterizing functions preserving positivity for
trees}\label{S3}

In this section we examine the effect the degree of a graph $G$ plays in
characterizing functions preserving positivity on $\bp_G$ when applied
entrywise. The simplest graph with a vertex of a given degree is a star
graph. Thus we begin by studying functions preserving positivity on
$\bp_G$ for star graphs $G$, and more generally, for $G$ a tree.

\subsection{Positive semidefinite matrices on star graphs}\label{Sstars}

Recall that a star graph has $d+1$ vertices for some $d \geq 0$, $d$
edges, and a unique vertex of degree $d$. The following result
characterizes positive semidefinite matrices with zeros according to a
star. Note that every nonempty graph contains a star subgraph, so the
result yields useful information about $\bp_G$ for all nonempty $G$, and
will be crucial in proving Theorem \ref{thmA}.

\begin{proposition}\label{thm:charac_pd_star}
Suppose $d \geq 0$ and
\begin{equation}\label{Eposdef}
A = \left(
\begin{BMAT}(b){c1ccc}{c1ccc}
p_1 & \alpha_2 & \cdots & \alpha_{d+1} \\
\alpha_2 & p_2 & & 0 \\
\vdots & & \ddots & \\
\alpha_{d+1} & 0 & & p_{d+1}
\end{BMAT}
\right)
\end{equation}

\noindent is a real-valued symmetric matrix with zeros according to a
star graph. Then $A$ is positive semidefinite if and only if the
following three conditions hold:
\begin{enumerate}
\item $p_i \geq 0$ for all $1 \leq i \leq d+1$;
\item for all $2 \leq i \leq d+1$, $p_i = 0 \implies \alpha_i = 0$; 
\item $\displaystyle p_1 \geq \sum_{\{i>1\ :\ p_i \neq 0\}} \alpha_i^2 /
p_i$.
\end{enumerate}
\end{proposition}

\begin{proof}
Let $A$ be as in Equation \eqref{Eposdef}. If $A \in \bp_{d+1}$, then (1)
and (2) are clear. To prove (3), define the function $h : (-\infty,0) \to
\R$, given by:
\[ h(\lambda) := \lambda - p_1 - \sum_{i=2}^{d+1}
\frac{\alpha_i^2}{\lambda - p_i} = \lambda - p_1 - \sum_{\{ 1 < i \leq m
\ : \ p_i > 0 \}}^{d+1} \frac{\alpha_i^2}{\lambda - p_i}, \]

\noindent by using (1) and (2). We now study if $h$ has a negative root,
which will lead to whether $A$ has a negative eigenvalue. Note that $h$
is well-defined since $\lambda < 0 \leq p_i$ for all $i$. It is also
clear that $h(\lambda) \to -\infty$ as $\lambda \to -\infty$. Moreover,
\[ h'(\lambda) = 1 - \sum_{\{ i>1 \ : \ p_i > 0 \}} \alpha_i^2 (-1)
(\lambda - p_i)^{-2} = 1 + \sum_{\{ i>1 \ : \ p_i > 0 \}}
\frac{\alpha_i^2}{(\lambda - p_i)^2} \geq 1. \]

\noindent Hence $h$ is strictly increasing. Note also that $h(\lambda)$
can be rewritten with the summation running over only those $i > 1$ such
that $p_i > 0$, by using (1) and (2). Then $h$ is continuous on
$(-\infty,0]$, and $h(0) = -p_0 + \sum_{i > 1\ :\ p_i > 0} \alpha_i^2 /
p_i$. We claim that this must be nonpositive, which shows (3).

Suppose by contradiction that the claim is false. Then by the
Intermediate Value Theorem for $h$, $h(\lambda_0) = 0$ for some
$\lambda_0 < 0$. We now claim that $A v = \lambda_0 v$ has a nonzero
solution $v'$, so that $Q_A(v') = \lambda_0 ||v'||^2 < 0$. Indeed, define
$v'_1 := 1$ and $v'_i := \frac{\alpha_i}{\lambda_0 - p_i}$ for $i>1$. It
is then easy to check that if $i>1$, then
\[ \alpha_i \cdot v'_1 + p_i v'_i = \alpha_i + \frac{p_i
\alpha_i}{\lambda_0 - p_i} = \frac{\lambda_0 \alpha_i}{\lambda_0 - p_i} =
\lambda_0 v'_i. \]

\noindent Moreover, for $i=1$,
\[ p_1 v'_1 + \sum_{\{ i>1 \ : \ p_i > 0 \}} \alpha_i v'_i = p_1 +
\sum_{\{ i>1 \ : \ p_i > 0 \}} \frac{\alpha_i^2}{\lambda_0 - p_i} =
\lambda_0 - h(\lambda_0) = \lambda_0 v'_1. \]

\noindent This proves that $A v' = \lambda_0 v'$, as desired. Hence $A$
is not positive semidefinite, which is a contradiction. This proves (3).
(Note that a similar argument could have been used to directly prove (2),
by considering $\lim_{t \to 0^-} h(t) = + \infty$ if $p_i = 0 \neq
\alpha_i$ for some $i>1$. In this case $h$ again has a negative root
$\lambda_0 < 0$, and the above choice of eigenvector again yields a
contradiction.)\bigskip

To show the converse, assume henceforth that (1)-(3) hold. Now define for
$m \in \N$:
\begin{equation}\label{Estar}
a_m := p_1^m - \sum_{\{ i > 1\ :\ p_i \neq 0 \}}
\frac{\alpha_i^{2m}}{p_i^m}, \qquad L_m := \left(
\begin{BMAT}(b){c1ccc}{c1ccc}
\sqrt{a_m} & \alpha_2^m p_2^{-m/2} & \cdots & \alpha_{d+1}^m
p_{d+1}^{-m/2} \\
0 & p_2^{m/2} & & 0 \\
\vdots & & \ddots & \\
0 & 0 & & p_{d+1}^{m/2}
\end{BMAT} \right),
\end{equation}

\noindent with the understanding (since (2) holds) that $\alpha_i^m
p_i^{-m/2}$ denotes $0$ if $p_i = 0$. Now since $a_1 \geq 0$ by (3),
$L_1$ is a real matrix, and it is easy to check that $A = L_1 L_1^T$.
This proves the converse, and hence the equivalence in the first part.
\end{proof}

\begin{corollary}\label{cor:eigenvalues}
If $A$ is as in \eqref{Eposdef}, then
\begin{equation}\label{Estardet}
\det A = \prod_{i=1}^{d+1} p_i - \sum_{i>1} \alpha_i^2 \prod_{j=2,\ j
\neq i}^{d+1} p_j.
\end{equation}

\noindent In particular, if $p_2 = p_3 = \dots = p_{d+1}$, then the
eigenvalues of $A$ are $p_2$ with multiplicity $d-1$, and the following
two eigenvalues with multiplicity one each (or multiplicity two if they
are equal):
\[ \frac{p_1 + p_2 \pm \sqrt{(p_1 - p_2)^2 + 4 \sum_{i=2}^{d+1}
\alpha_i^2}}{2}. \]
\end{corollary}

\begin{proof}
It is clear that if $p_2, \dots, p_{d+1} > 0$ and $a_1 > 0$, then $\det A
= \det L_1 L_1^T = (\det L_1)^2$ by Proposition \ref{thm:charac_pd_star},
where $L_1$ was defined in Equation \eqref{Estar}. Note that $(\det
L_1)^2$ is precisely the claimed expression \eqref{Estardet}. Now the
determinant is a polynomial in the $2d+1$ entries $p_1, p_i, \alpha_i$
(for $2 \leq i \leq d+1$), which equals the polynomial expression
\eqref{Estardet} for a Zariski dense subset of $\R^{2d+1}$. Hence it
equals the polynomial \eqref{Estardet} at all points in $\R^{2d+1}$.
Finally, to determine the eigenvalues when $p_2 = \dots = p_{d+1}$,
compute the characteristic polynomial $\det(A - \lambda \Id_n)$ using
Equation \eqref{Estardet}, and solve for $\lambda$.
%
\end{proof}

Theorems \ref{Therz} and \ref{Tvasu} characterize functions mapping
$\bp_n(I)$ into $\bp_{K_n}$ for every $n \geq 1$. Before proceeding to
study the case of trees, it is natural to ask which functions $f$ map
$\bp_n$ into $\bp_G$ when $G$ is a non-complete graph on $n$ vertices.
Proposition \ref{thm:f(0)} below shows that such functions have to satisfy
many restrictions. In particular, when $I = (-R,R)$ for some $R > 0$, the
only such function is $f \equiv 0$.

\begin{proposition}\label{thm:f(0)}
Suppose $0 \in I \subset \R$ is an interval with $\sup I \not\in I$ and
$|\inf I| \leq \sup I$. Let $G$ be a graph and $f: I \to \R$ such that $f
\not\equiv 0$. Suppose $f_G[-]$ sends all of $\bp_{|G|}(I)$ to $\bp_G$.
Then every connected component of $G$ is complete.
\end{proposition}

\noindent Note that the condition $|\inf I| \leq \sup I$ is assumed in
Theorem \ref{thm:vasudeva:M2} because no $2 \times 2$ matrix in
$\bp_2(I)$ can have any entry in $(-\infty, -\sup I)$.

\begin{proof}
Suppose $f_G[-]$ sends all of $\bp_{|G|}(I)$ to $\bp_G$. Assume to the
contrary that not every component of $G$ is complete. Then, without loss
of generality, $(1,2), (1,3) \in E$ but $(2,3) \notin E$. Suppose $a \in
I \cap [0,\infty)$; since $B := a {\bf 1}_{|G| \times |G|} \in
\bp_{|G|}(I)$, hence the principal $3 \times 3$ submatrix of $f_G[B]$ is
in $\bp_3$. But this is precisely the matrix $f(a) B(1,1,1)$, where
\begin{equation}\label{eqn:B_3}
B(\mu,\alpha,\beta) := \begin{pmatrix} \mu & \alpha & \beta\\ \alpha &
\alpha & 0\\ \beta & 0 & \beta\end{pmatrix}, \qquad \mu, \alpha, \beta
\in \R.
\end{equation}

\noindent Thus, the diagonal entries and determinant of $f(a) B(1,1,1)$
must be nonnegative; this yields $f(a) \geq 0$ and $-f(a)^3 \geq 0$.
Therefore $f(a) = 0$ for every $a \in I \cap [0,\infty)$. Now if $a \in
I$ is negative, apply $f_G[-]$ to the matrix $\begin{pmatrix} |a| & a\\ a
& |a| \end{pmatrix} \oplus {\bf 0}_{(|G|-2) \times (|G|-2)} \in
\bp_G(I)$, and consider the leading principal $2 \times 2$ submatrix.
Since $f(|a|) = 0$ from above, hence $f(a) = 0$ as well, which
contradicts the assumption that $f \not\equiv 0$.
\end{proof}

\begin{remark}
Applying Proposition \ref{thm:f(0)} with $f(x) \equiv x$ and any interval
$0 \in I \subset \R$ shows that $f_G[-]$ does not send all of
$\bp_{|G|}(I)$ to $\bp_G$ if $G$ is not a union of disconnected complete
components. In other words, thresholding according to a non-complete
connected graph, an important procedure in applications in
high-dimensional probability and statistics, does not preserve positive
definiteness (see \cite[Theorem 3.1]{Guillot_Rajaratnam2012}).
\end{remark}

\subsection{Characterization for trees}

We now use the analysis in Section \ref{Sstars} to show Theorem
\ref{thmA}. We first need the following preliminary result. 

\begin{proposition}\label{Pf(0)}
Suppose $0 \in I \subset \R$ is an interval with $\sup I \not\in I$ and
$|\inf I| \leq \sup I$. Let $G$ be a non-complete connected graph and $f:
I \to \R$. If $f_G[-]$ sends $\bp_G(I)$ to $\bp_G$, then $f(0) = 0$ and
$f$ is superadditive on $I \cap (0,\infty)$. 
\end{proposition}

\begin{proof}
Suppose without loss of generality that $V(G) = \{ 1, 2, 3, \dots,n\}$,
$(1,2), (1,3) \in E(G)$, but $(2,3) \not\in E(G)$. Applying $f_G[-]$ to
the matrix $\bf{0}_{|G| \times |G|}$ shows that $f(0) B(1,1,1)$ (defined
in Equation \eqref{eqn:B_3}) is positive semidefinite. This is only
possible if $f(0) = 0$.

We now show that $f(\alpha + \beta) \geq f(\alpha) + f(\beta)$ whenever
$\alpha, \beta, \alpha + \beta \in I$. This is clear if either $\alpha$
or $\beta$ is zero, since $f(0) = 0$; so we now assume that $\alpha,
\beta > 0$. By Theorem \ref{thm:vasudeva:M2}, we may also assume that
$f(x) > 0$ on $I \cap (0,\infty)$. Then $f_G[B(\alpha+\beta,\alpha,\beta)
\oplus {\bf 0}_{(|G|-3) \times (|G|-3)}] \in \bp_G$. Recall that
$f(\alpha), f(\beta), f(\alpha + \beta) > 0$ by Theorem
\ref{thm:vasudeva:M2}. Now applying Proposition \ref{thm:charac_pd_star}
to the leading principal $3 \times 3$ submatrix of
$f_G[B(\alpha+\beta,\alpha,\beta) \oplus {\bf 0}_{(|G|-3) \times
(|G|-3)}]$, we obtain that $f(\alpha + \beta) \geq f(\alpha) + f(\beta)$,
which concludes the proof.
\end{proof}

\begin{remark}
Proposition \ref{Pf(0)} shows that if $f_G[-]$ maps $\bp_G$ into itself,
then $f(0) = 0$; as a consequence, $f_G[-]$ reduces to the standard
entrywise function $f[-]$.
\end{remark}
 
We can now prove Theorem \ref{thmA}. 

\begin{proof}[Proof of Theorem \ref{thmA}]
Clearly $(2) \Rightarrow (1) \Rightarrow (3)$. We now prove that $(3)
\Rightarrow (4)$ and $(4) \Rightarrow (2)$. \medskip

\noindent {\bf (3) $\Rightarrow$ (4).}
If $f \equiv 0$ on $I$ then the result is obvious. Now assume $f_{A_3}[A]
\in \bp_{A_3}$ for every $A \in \bp_{A_3}(I)$. In particular, $f[A] \in
\bp_{K_2}$ for every $A \in \bp_{K_2}(I)$. Therefore, by Theorem
\ref{thm:vasudeva:M2}, $f$ satisfies \eqref{Emidconvex} on $I$. Now
consider the matrix $A$ in Equation \eqref{Eposdef} for $d=2$. By
Proposition \ref{thm:charac_pd_star}, for $0 < p_i, \alpha_i \in I$, we
have $A \in \bp_{A_3}(I)$ if and only if $p_1 \geq \alpha_2^2/p_2 +
\alpha_3^2/p_3$.
Now suppose $0 < \alpha_2, \alpha_3, \alpha_2 + \alpha_3 \in I$; then
$f(\alpha_2), f(\alpha_3) > 0$ by Theorem \ref{thm:vasudeva:M2}. Now
$B(\alpha_2+\alpha_3,\alpha_2,\alpha_3)$ (defined in Equation
\eqref{eqn:B_3}) lies in $\bp_{A_3}(I)$, so
$f_{A_3}[B(\alpha_2+\alpha_3,\alpha_2,\alpha_3)] \in \bp_{A_3}$. Thus, by
Proposition \ref{thm:charac_pd_star}, 
\[ f(p_1) = f(\alpha_2 + \alpha_3) \geq \frac{f(\alpha_2)^2}{f(p_2)} +
\frac{f(\alpha_3)^2}{f(p_3)} = f(\alpha_2) + f(\alpha_3). \]

\noindent This proves $f$ is superadditive. The case when $\alpha_2$ or
$\alpha_3$ is zero follows from Proposition \ref{Pf(0)}.\medskip

\noindent{\bf (4) $\Rightarrow$ (2).}

\noindent Once again, if $f \equiv 0$ on $I$ then the result is
immediate. Now suppose $f$ is superadditive, not identically zero on $I$,
and satisfies \eqref{Emidconvex} on $I$. Let $0 \leq y < x \in I$. Then
$x-y \in (0,x] \subset I$, so by the superadditivity of $f$,
\[ f(x) = f(y + x-y) \geq f(y) + f(x-y) \geq f(y). \]

\noindent Moreover, if $0 \in I$, then $0 \leq f(0) \geq f(0) + f(0)$ by
super-additivity, so $f(0) = 0$. This shows that $f$ is nonnegative and
nondecreasing on $I$. Hence by Theorem \ref{thm:vasudeva:M2}, $f[A] \in
\bp_{K_2}$ for every $A \in \bp_{K_2}((0, \infty))$.

Now since $f \not\equiv 0$ on $I$, hence $f(p) > 0$ for all $0 < p \in I$
by Theorem \ref{thm:vasudeva:M2}. Moreover, Equation \eqref{Emidconvex}
trivially holds if $x$ or $y$ is zero (and $0 \in I$). Now assume that
$x,y > 0$; then \eqref{Emidconvex} can be restated as:
\begin{equation}\label{eqn:K2}
p, \frac{\alpha^2}{p} \in I, \ p>0 \quad \implies \quad
f\left(\frac{\alpha^2}{p}\right) \geq \frac{f(\alpha)^2}{f(p)}.
\end{equation}

We now prove that (2) holds for any tree $T$ by induction on $|T| \geq
3$. Suppose first that $T$ is a tree with $3$ vertices, i.e., $T = A_3$.
Then, by Proposition \ref{thm:charac_pd_star}, $f_{A_3}[A] \in \bp_{A_3}$
for every $A \in \bp_{A_3}$ if and only if
\begin{equation}\label{eqn:n3}
f\left(\frac{\alpha_2^2}{p_2} + \frac{\alpha_3^2}{p_3}\right) \geq
\frac{f(\alpha_2)^2}{f(p_2)} + \frac{f(\alpha_3)^2}{f(p_3)},
\end{equation}

\noindent (or if one of $p_2, p_3$ is zero, in which case the assertion
is easy to verify). Now suppose $0 < p_2, p_3 \in I$. If $A \in
\bp_{A_3}(I)$, then $p_1 \in I$, so $\frac{\alpha_2^2}{p_2} +
\frac{\alpha_3^2}{p_3} \in [0,p_1]$ is also in $I$. Hence \eqref{eqn:n3}
follows immediately by the superadditivity of $f$ and by \eqref{eqn:K2}.

Therefore $(4) \Rightarrow (2)$ holds for a tree with $n=3$ vertices. Now
assume that $A \in \bp_{T'}(I)$ implies $f_{T'}[A] \in \bp_{T'}$ for any
tree $T'$ with $n$ vertices, and consider a tree $T$ with $n+1$ vertices.
Let $\widetilde{T}$ be a sub-tree obtained by removing a vertex connected
to only one other node. Without loss of generality, assume the vertex
that is removed is labeled $n+1$ and its neighbor is labeled $n$. Let $A
\in \bp_T(I)$; then $A$ has the form 
\[ A = \left(\begin{array}{ccc}
\widetilde{A}_{n \times n} & & \mc{{\bf 0}_{(n-1) \times 1}} \\
& & \mc{a} \\
\cline{1-3}
{\bf 0}_{1 \times (n-1)} & a &  \mc{\alpha}
\end{array}\right). \]

\noindent If $\alpha = 0$ then $a = 0$ since $A$ is positive
semidefinite, and thus $f_T[A] \in \bp_G$ since $f(0) = 0$. When $\alpha
\not= 0$, the Schur complement $S_A$ of $\alpha$ in $A$ is $S_A =
\widetilde{A} - (a^2/\alpha) E_{n,n}$.
Here, $E_{i,j}$ denotes the $n \times n$ elementary matrix with the
$(i,j)$ entry equal to $1$, and every other entry equal to $0$. Since $A
\in \bp_{T}(I)$, hence $\widetilde{A} \in \bp_{\widetilde{T}}(I)$, and
$S_A \in \bp_{\widetilde{T}}(I)$ from the above analysis (since
$(S_A)_{nn} = \widetilde{a}_{nn} - a^2/\alpha \in [0,\widetilde{a}_{nn})
\subset I$). Therefore, by the induction hypothesis,
$f_{\widetilde{T}}[\widetilde{A}], f_{\widetilde{T}}[S_A] \in
\bp_{\widetilde{T}}$. Consider now the matrix $f_T[A]$. Using Schur
complements, $f_T[A] \in \bp_T$ if and only if
$f_{\widetilde{T}}[\widetilde{A}] \in \bp_{\widetilde{T}}$ and the Schur
complement $S_{f_T[A]}$ of $f(\alpha) > 0$ in $f_T[A]$, given by
\[ S_{f_T[A]} = f_{\widetilde{T}}[\widetilde{A}] -
\frac{f(a)^2}{f(\alpha)} E_{n,n}, \]

\noindent belongs to $\bp_{\widetilde{T}}$. Now, notice that 
$f_{\widetilde{T}}[S_A] = f_{\widetilde{T}}[\widetilde{A}] + \left[
f(b) - f(\widetilde{a}_{nn}) \right] E_{n,n}$,
where $b := (S_A)_{nn} = \widetilde{a}_{nn} - \frac{a^2}{\alpha} \in I$
from the above analysis. Since $f_{\widetilde{T}}[S_A] \in
\bp_{\widetilde{T}}$ from above, to conclude the proof, it suffices to
show that 
\begin{equation}\label{eqn:schur_compl}
- \frac{f(a)^2}{f(\alpha)} \geq f(b) - f(\widetilde{a}_{nn}). 
\end{equation}

\noindent Indeed, by using the superadditivity of $f$ and \eqref{eqn:K2},
we compute:
\[ f(\widetilde{a}_{n,n}) = f\left(\frac{a^2}{\alpha} + b \right) \geq
f\left(\frac{a^2}{\alpha}\right) + f(b) \geq \frac{f(a)^2}{f(\alpha)} +
f(b), \]

\noindent which proves \eqref{eqn:schur_compl}. Therefore $(4)
\Rightarrow (2)$ holds for a tree with $n+1$ vertices. This completes the
induction and the proof of the theorem.  
\end{proof}

\begin{remark}
Hiai suggests in \cite[Remark 3.4]{Hiai2009} that optimal conditions for
$f$ to preserve $\bp_3(-R,R)$ for $0 < R \leq \infty$ could be that $f$
is continuous on $(-R,R)$. However, note from Theorem \ref{thmA} that any
such $f$ for which $f(0) = 0$, also preserves $\bp_{A_3}([0,R))$, and
hence is necessarily continuous, nondecreasing, positive, super-additive,
and satisfies \eqref{Emidconvex} on $(0,R)$. These conditions place
severe restrictions on the set of admissible $f$ preserving
$\bp_3(-R,R)$. 
\end{remark}

\begin{corollary}
Let $I = [0,R)$ for some $0 < R \leq \infty$. Let $f:I
\rightarrow \mathbb{R}$ and assume $f_G[A] \in \bp_G$ for every $A \in
\bp_G(I)$ for some non-complete connected graph with at least $3$
vertices. Then $f$ is superadditive and multiplicatively mid-point convex
(see \eqref{Emidconvex}).
\end{corollary}

\begin{proof}
The proof follows by noticing that $G$ contains a copy of $A_3$ as an
induced subgraph.  
\end{proof}

\section{Fractional Hadamard powers and absolute monotonicity}\label{S4}

Recall from Theorem \ref{thmA} that general functions preserving
positivity on $\bp_G$ for a tree $G$ are necessarily multiplicatively
mid-point convex and superadditive. We now explore a special sub-family
of these functions in greater detail: the power functions $x^\alpha$. We
do so for various reasons: first, recall that by the Schur product
theorem, every integer entrywise power of a positive semidefinite matrix
is positive semidefinite. Studying which powers $\alpha > 0$ preserve
Loewner positivity on $\bp_G$ for non-complete graphs $G$ is a natural
extension of this problem.
Additionally, power functions are natural to study since they are
tractable as compared to more general families of functions. Finally,
there are also precedents in the literature for studying power functions
preserving positivity; see e.g.~\cite{Bhatia-Elsner,
FitzHorn,GKR-crit-2sided,Hiai2009}. The following important result
characterizes the powers preserving positivity for symmetric matrices
with nonnegative entries.

\begin{theorem}[FitzGerald and Horn, {\cite[Theorem
2.2]{FitzHorn}}]\label{thm:fitz_horn_fractional}
Suppose $A \in \bp_n([0,\infty))$ for some $n \geq 2$, and $\alpha \geq
n-2$. Then $A^{\circ\alpha} := ((a_{ij}^\alpha))_{i,j} \in \bp_n$. If
$\alpha \in (0,n-2)$ is not an integer, then there exists $A \in
\bp_n((0,\infty))$ such that $A^{\circ\alpha} \notin \bp_n$.
\end{theorem}

A natural generalization of the aforementioned problem would be to
characterize the powers preserving positivity for matrices with zeros
according to a graph. Using Theorem \ref{thmA}, we now prove an analogue
of Theorem \ref{thm:fitz_horn_fractional} for $\bp_G$ when $G$ is a tree. 

\begin{proposition}\label{thm:frac_tree}
Let $G$ be a tree with $n \geq 3$ vertices. Suppose $A \in
\bp_G([0,\infty))$. Then $A^{\circ \alpha} := ((a_{ij}^\alpha))_{i,j} \in
\bp_G$ for every $\alpha \geq 1$. If $0 < \alpha < 1$ and $0 < R \leq
\infty$, then there exists $A_R \in \bp_G([0,R))$ such that $A_R^{\circ
\alpha} \not\in \bp_G$.
\end{proposition}

\begin{proof}
Say $f(x) := x^\alpha$. By Theorem \ref{thmA}, $f[-]$ preserves
positivity on $\bp_G([0,\infty))$ if and only if it preserves positivity
on $\bp_{A_3}([0,\infty))$, which by Proposition \ref{thm:charac_pd_star}
holds if and only if for every $p_1, p_2, p_3 \geq 0$ and every
$\alpha_2, \alpha_3 > 0$,
\[ p_1 \geq \frac{\alpha_2^2}{p_2} + \frac{\alpha_3^2}{p_3} \quad
\Rightarrow \quad f(p_1) \geq \frac{f(\alpha_2)^2}{f(p_2)} +
\frac{f(\alpha_3)^2}{f(p_3)}. \]

\noindent Since $f$ is increasing on $(0,\infty)$, the previous condition
is equivalent to 
\[ f \left( \frac{\alpha_2^2}{p_2} + \frac{\alpha_3^2}{p_3} \right) \geq
\frac{f(\alpha_2)^2}{f(p_2)} + \frac{f(\alpha_3)^2}{f(p_3)}, \]

\noindent which holds for the multiplicative function $f(x) = x^\alpha$,
if and only if $\alpha \geq 1$. This proves the result when $\alpha \geq
1$, while for $\alpha < 1$, it implies that there exists $A \in
\bp_G([0,\infty))$ such that $A^{\circ \alpha} \notin \bp_G$. Rescaling
$A$ by a small enough constant $c_R > 0$ such that $c_R A \in
\bp_G([0,R))$, we obtain the desired counterexample $A_R := c_R A \in
\bp_G([0,R))$.
\end{proof}

Recall from Section \ref{S2} that characterizing entrywise functions
preserving positivity in a fixed dimension is a difficult problem.
Theorem \ref{thm:fitz_horn_fractional} provides a large family of
functions mapping $\bp_n([0,\infty))$ into itself, for any $n \geq 1$.
Namely, given a nonnegative measure $\mu_n$ on $[n-2,\infty)$, the
function
\begin{equation}\label{eqn:f_mu}
f^{\mu_n}(x) := \sum_{i=1}^{n-3} a_i x^i + \int_{n-2}^\infty x^\alpha\ d
\mu_n(\alpha), \qquad x>0,
\end{equation}

\noindent preserves $\bp_n([0,\infty))$ for all choices of nonnegative
scalars $a_1, \dots, a_{n-3}$ (see \cite[Corollary 2.3]{FitzHorn}). In
particular, if one imposes the condition that $f[-]$ preserves
$\bp_n([0,\infty))$ for all $n$ (or equivalently $\bp_{K_n}([0,\infty))$
for all $n$), then the intersection of the above families over all $n >
2$ is precisely the set of absolutely monotonic functions; see Theorem
\ref{Therz}. Given the above observations, it is natural to ask if every
function $f[-] : \bp_n([0,\infty)) \to \bp_n$ is necessarily of the form
\eqref{eqn:f_mu}. Note that this is indeed the case if one imposes rank
constraints on $f$; see Theorem \ref{Tlowrank}.

Similarly, if $G$ is a tree with $n \geq 3$ vertices, Proposition
\ref{thm:frac_tree} implies that for any nonnegative measure $\mu$ on
$[1,\infty)$, functions of the form
\begin{equation}\label{Efmu}
f^\mu(x) := \int_1^\infty x^\alpha\ d\mu(\alpha), \qquad x > 0,
\end{equation}

\noindent map $\bp_G([0,\infty))$ into itself. We ask if every function
preserving $\bp_G([0,\infty))$ has to be of this form. Theorem
\ref{thmB} provides a negative answer to these questions. First, note
that entrywise functions mapping $\bp_n$ into itself are not necessarily
of the form \eqref{eqn:f_mu} when $n=2$ since by Theorem \ref{thmB},
there exists an analytic function $f$ with some negative coefficients,
which maps $\bp_2([0,\infty))$ into $\bp_2$. More generally, Theorem
\ref{thmB} provides an example of a function not of the form \eqref{Efmu}
that map $\bp_T([0,\infty))$ into $\bp_T$ for all trees $T$.

\subsection{Proof of Theorem \ref{thmB}}

We now proceed to prove the second main result of this paper. The proof
requires constructing and working with multiplicatively convex
polynomials with negative coefficients. We first collect together some
basic properties of these functions.

\begin{definition}
Given an interval $I \subset [0,\infty)$, a function $f: I \rightarrow
[0,\infty)$ is said to be \emph{multiplicatively convex} if
$f(x^{1-\lambda}y^\lambda) \leq f(x)^{1-\lambda} f(y)^\lambda$ for all
$x,y \in I$ and $0 \leq \lambda \leq 1$.
(Here we set $0^0 = 1$.)
\end{definition}

Clearly, a function $f$ is multiplicatively convex if and only if $\log
f$ is a convex function of $\log x$, i.e., the function $g(x) = \log
f(e^x)$ is convex.

\begin{theorem}[Properties of multiplicatively convex
functions, \cite{niculescu}]\label{Tmc}
Let $I \subset [0,\infty)$ be an interval, and $f, g: I \rightarrow
[0,\infty)$.
\begin{enumerate}
\item If $f,g$ are multiplicatively convex, then so are $f+g, fg, \alpha
f$ for all $0 \leq \alpha \in \R$. In particular, every polynomial with
nonnegative coefficients is multiplicatively convex.

\item $f[A]$ is positive semidefinite for every $A \in \bp_2(I)$ of rank
$1$, if and only if
\begin{equation}\label{eqn:logconvex} 
f(\sqrt{xy})^2 \leq f(x) f(y) \qquad \forall x,y \in I.
\end{equation}

\item $f[A]$ is positive semidefinite for every $A \in \bp_2(I)$ if and
only if $f$ satisfies \eqref{eqn:logconvex} and is nondecreasing on $I$. 

\item If $0 \notin I$ and $f$ is continuous, then $f$ satisfies
\eqref{eqn:logconvex} if and only if $f$ is multiplicatively convex.

\item If $I$ is open and $f$ is twice differentiable on $I$, then $f$ is
multiplicatively convex on $I$ if and only if
\begin{equation}\label{Epsi}
\Psi_f(x) := x \left[f''(x) f(x) - (f'(x))^2\right] + f(x)f'(x) \geq 0
\quad \forall x \in I.
\end{equation}
\end{enumerate}
\end{theorem} 

\noindent These properties are all proved in \cite{niculescu}. The first
part follows from Exercises 2.1.3, 2.1.4, and Proposition $2.3.3$ in {\it
loc.~cit.} (the last is attributed to Hardy, Littlewood, and P\'olya).
The second part is obvious, while the third part follows from Theorem
\ref{thm:vasudeva:M2}. The fourth and fifth parts follow from Theorem
$2.3.2$ and Exercise $2.4.4$ in \cite{niculescu} respectively.

Note that by continuity, a polynomial $p$ is multiplicatively convex if
and only if it satisfies \eqref{eqn:logconvex}. If in addition, $p$ takes
only positive values on $(0,\infty)$, then its first and last
coefficients are necessarily positive.

\begin{proposition}\label{Pposcoeff}
Let $p(x) = \sum_{k=0}^n a_k x^k$ be a polynomial of degree $n \geq 3$.
Assume $p(x) > 0$ for every $x > 0$ and $p$ satisfies
\eqref{eqn:logconvex} on $(0,\infty)$. Then $a_0, a_1, a_{n-1}, a_n \geq
0$.
\end{proposition}

\begin{proof}
Since $p(x) > 0$ for every $x > 0$, then $a_0, a_n > 0$.  Now consider
\eqref{eqn:logconvex} with $y = x/2$. Then,
\[ q(x) := p(x^2)p(x^2/4) - p(x^2/2)^2 = \frac{a_{n-1}a_n}{4^n}x^{4n-2} +
\dots + \frac{a_0a_1}{4} x^2, \]

\noindent where only the lowest and highest order terms are displayed.
Since $q(x) \geq 0$ for every $x > 0$, then $a_n a_{n-1} \geq 0$ and $a_0
a_1 \geq 0$. Since $a_0, a_n > 0$, then it follows that $a_{n-1}, a_n
\geq 0$.
\end{proof}

We now show that Proposition \ref{Pposcoeff} is the best possible result
along these lines, in the sense that apart from the first two and last
two coefficients, {\it every} other coefficient of a positive
multiplicatively convex polynomial can be negative.

\begin{theorem}\label{Tnon-abs-mon}
Fix $0 < r < s < \infty$, $B \subset (r,s)$, and $a_r, a_s > 0$. Now let 
\begin{equation}\label{Efff}
f(x) = a_r x^r + a_s x^s + \int_B h(\beta) x^\beta\ d \mu(\beta),
\end{equation}
where $\mu$ is a nonnegative measure on $B$ such that $\mu(B) > 0$, and
$h : B \to \R$ is such that $\beta \mapsto h(\beta) x^\beta$ is
$\mu$-measurable on $B$.
\begin{enumerate}
\item Suppose $r > 1$. Then there exists $\nu > 0$ such that if $h(\beta)
> -\nu\ \forall \beta \in B$, then $f(x)$ is nonnegative and
super-additive on $[0,R)$.

\item Suppose $0 \leq r' <  r < s < s'$, and let $a_{r'}, a_{s'} > 0$.
Then there exists $\lambda > 0$ such that if $h(\beta) > -\lambda\
\forall \beta \in B$, then $g(x) := f(x) + a_{r'} x^{r'} + a_{s'} x^{s'}$
is multiplicatively convex on $[0,R)$.
\end{enumerate}
\end{theorem}

\begin{proof}
Define for each $\beta \in B$:
\begin{equation}\label{E11}
f_\beta(x) := \frac{a_r x^r + a_s x^s}{\mu(B)} + h(\beta) x^\beta, \qquad
g_\beta(x) := f_\beta(x) + \frac{a_{r'} x^{r'} + a_{s'} x^{s'}}{\mu(B)}.
\end{equation}

\noindent It is clear that sums and integrals of super-additive functions
are super-additive. Thus, if $f_\beta$ is super-additive on $[0,R)$
whenever $h(\beta) > -\nu$, then so is
\[ \int_B f_\beta(x)\ d \mu(\beta) \equiv f(x). \]

Similarly, we claim that multiplicatively convex functions are closed
under taking sums and integrals. Indeed, simply note that $g : [0,R) \to
\R$ is multiplicatively convex if and only if $g[A] \in \bp_2$ for all $A
\in \bp_2([0,R))$. Therefore, it suffices to prove the second part of the
theorem for functions of the form $g_\beta$.\medskip

\noindent \textbf{Proof of (1).}
Suppose as in Equation \eqref{E11} that $f(x) = c_r x^r + c_s x^s +
c_\beta x^\beta$ for some $1 < r < \beta < s < \infty$, and where $c_r,
c_s > 0$. We show the result in this special case, when $R = \infty$.
Define
\[ \nu' := \frac{r(r-1)}{s(s-1)} \min(c_r, c_s). \]

\noindent Note that if $c_\beta \geq 0$ then the function $f$ is clearly
nonnegative and super-additive on $[0,\infty)$. Suppose now that $-\nu' <
c_\beta < 0$. Observing that $x^{\beta-1} < x^{r-1} + x^{s-1}$ for all $x
\geq 0$, we compute:
\[ -\beta c_\beta x^{\beta - 1} < \beta |c_\beta| (x^{r-1} + x^{s-1}) < s
\nu' (x^{r-1} + x^{s-1}) \leq r c_r x^{r-1} + s c_s x^{s-1}, \quad
\forall x \geq 0. \]

\noindent We conclude that $f(x)$ is strictly increasing on $(0,\infty)$.
Since $f(0) = 0$, it is also positive on $(0,\infty)$.

We now claim that when $c_r, c_s > 0$, the function $f(x) = c_r x^r + c_s
x^s + c_\beta x^\beta$ is also super-additive on $[0,\infty)$ when $-\nu'
< c_\beta$. We may assume that $c_\beta \in (-\nu',0)$ since otherwise
the assertion is clear. To show the claim, we first make some
simplifications. Note that since $f(0) = 0$, a reformulation of
superadditivity is that $\Delta_h f : [0,\infty)$ is minimized at $0$ for
all $h>0$. Here $(\Delta_h f)(x) := f(x+h)-f(x)$. In particular, $f$ is
superadditive on $[0,\infty)$ if for all $h>0$, the function $(\Delta_h
f)(x)$ is nondecreasing for $x \in [0,\infty)$. Since $f$ is smooth on
$(0,\infty)$, this latter condition is equivalent to saying that
$\Delta_h(f')(x) \geq 0$ for all $x,h > 0$. In turn, this follows if
$f''$ is nonnegative on $(0,\infty)$, by the Mean Value Theorem. Now note
that if $x>0$, then
\begin{align*}
f''(x) = &\ x^{-2} \left( r(r-1) c_r x^r + \beta(\beta-1) c_\beta x^\beta
+ s(s-1) c_s x^s \right)\\
\geq &\ x^{-2} \left( r(r-1) c_r x^r - s(s-1) \nu' x^\beta + s(s-1) c_s
x^s \right)\\
\geq &\ s(s-1) x^{-2} \left( \nu' x^r - \nu' x^\beta + \nu' x^s \right) =
s(s-1) \nu' x^{-2}(x^r + x^s - x^\beta) \geq 0,
\end{align*}

\noindent where we used the definition of $\nu'$, and also that $1 < r <
\beta < s$. Therefore by the above analysis, $f$ is superadditive on
$(0,\infty)$ if $c_\beta > -\nu'$. In the general case, one would set
$\nu := \mu(B)^{-1} \nu'$.\medskip

\noindent \textbf{Proof of (2).}
Suppose as in Equation \eqref{E11} that
\[ g(x) = c_{r'} x^{r'} + c_r x^r + c_\beta x^\beta + c_s x^s + c_{s'}
x^{s'}, \]

\noindent with $0 \leq r' < r < s < s' < \infty$ and $c_r, c_s, c_{r'},
c_{s'} > 0$.
By Theorem \ref{Tmc}(4), it is obvious that $x^\beta$ is multiplicatively
convex on $[0,\infty)$ for all $\beta \geq 0$. Hence if $c_\beta \geq 0$,
then $g(x)$ is multiplicatively convex by Theorem \ref{Tmc}(1). Thus,
suppose for the remainder of the proof that $c_\beta < 0$. We now use
Theorem \ref{Tmc} to show that $g$ is multiplicatively convex on
$[0,\infty)$ if $c_\beta \in (-\lambda,0)$ for some $\lambda>0$. To do
so, we need to compute $\Psi_g(x)$ (see Equation \eqref{Epsi}) and obtain
an expression for $\lambda$ using the previous part. The computation of
$\Psi_g$ can be carried out in greater generality: suppose $T \subset \R$
is a countable subset such that the addition map $: T \times T \to \R$
has finite fibers. Now if $g(x) = \sum_{t \in T} c_t x^t$ is defined for
$x$ in an open interval, then using the fact that $g$ is a homogeneous
linear polynomial in the $c_t$ (and hence $\Psi_g$ is homogeneous
quadratic),
\[ \Psi_g(x) = \sum_{t \neq t' \in T} c_t c_{t'} (t-t')^2 x^{t+t'-1}.
\]

\noindent Returning to the specific $g$ above, $\Psi_g(x)$ has lowest
degree term $c_r c_{r'} x^{r+r'-1}$ and highest degree term $c_s c_{s'}
x^{s+s'-1}$. Hence by the proof of the previous part, $x \Psi_g(x)$, and
hence $\Psi_g$, are positive on $(0,\infty)$, if all ``intermediate"
negative coefficients are bounded below by a threshold, say $\nu''$. But
these coefficients are precisely $c_r c_\beta, c_s c_\beta, c_{r'}
c_\beta, c_{s'} c_\beta$. Finally, define
\[ \lambda := \max(c_r, c_s, c_{r'}, c_{s'})^{-1} (s'-r')^{-2} \nu''. \]

\noindent Now if $-\lambda < c_\beta < 0$, then a typical negative
coefficient in $\Psi_g(x)$ is of the form 
\[ - c_\beta c_r (r-\beta)^2 \leq -c_\beta \max(c_r, c_s, c_{r'}, c_{s'})
(s'-r')^2 < \lambda \max(c_r, c_s, c_{r'}, c_{s'}) (s'-r')^2 \leq \nu'',
\]

\noindent which proves the result.
\end{proof}

Using Theorem \ref{Tnon-abs-mon}, we can now construct classes of
polynomials with negative coefficients such that the polynomial and its
derivatives are increasing, super-additive, or multiplicatively convex. 

\begin{corollary}\label{Cnon-abs-mon}
Suppose $p(x) = x^{m+1} \sum_{k=0}^n a_k x^k$ for some $m,n \in \N$.
Assume $a_0, a_n > 0$ and let $I := \{ 0 < k < n : a_k < 0 \}$.
\begin{enumerate}
\item There exists $\nu > 0$ such that if $-\nu < a_k < \infty$ for all
$k \in I$, then $p(x), p'(x), \dots, p^{(m-1)}(x)$ are strictly
increasing on $[0,\infty)$.

\item There exists $\lambda > 0$ such that if $-\lambda < a_k < \infty$
for all $k \in I$, then $p(x), p'(x), \dots, p^{(m-1)}(x)$ are
super-additive on $[0,\infty)$.

\item Suppose $n>2$ and $a_1, a_{n-1}$ are also positive. Then there
exists $\eta > 0$ such that if $-\eta < a_k < \infty$ for all $k \in I$,
then $p(x), p'(x), \dots, p^{(m)}(x)$ are multiplicatively convex on
$[0,\infty)$.
\end{enumerate}
\end{corollary}

\begin{proof}
The first two parts follow by applying Theorem \ref{Tnon-abs-mon} (with
$h \equiv 0$ or $B = \emptyset$, and $b_i \in \N$ for all $i$) to each of
$p, p', \dots, p^{(m-1)}$, and considering the intersection of all such
intervals. The third part follows by applying the theorem to each of $p,
p', \dots, p^{(m)}$.
\end{proof}

Using the above analysis, we can now prove Theorem \ref{thmB}.

\begin{proof}[Proof of Theorem \ref{thmB}]
By Theorem \ref{thmA}, it suffices to construct an entire function $f(z)
= \sum_{n=0}^\infty a_n z^n$ such that
(1) $a_n \in [-1,1]$,
(2) the sequence $(a_n)_{n \geq 0}$ contains arbitrarily long strings of
negative numbers,
(3) $f$ is nonnegative on $[0,\infty)$, and
(4) $f$ is multiplicatively convex and super-additive on $[0,\infty)$.
To construct such a function, let $q_n \geq n+4$ be a sequence of
increasing integers and let $r_n = \sum_{k=1}^n q_k$. By Corollary
\ref{Cnon-abs-mon}, for every $n \geq 1$ there exists a polynomial
$p_n(x) =x^{r_n}\sum_{k=0}^{n+3} a_{k,n} x^k$
satisfying properties (3) and (4), and such that $p_n$ is increasing on
$[0,\infty)$ and $a_{k,n} < 0$ for $2 \leq k \leq n+1$. Without loss of
generality, we can also assume that the coefficients of $p_n$ also belong
to the interval $[-1,1]$ for all $n \geq 1$. Now define
\[ f(z) := \sum_{n=1}^\infty \frac{p_n(z)}{(r_n+n+3)!} \qquad (z \in
\mathbb{C}). \]
Clearly, the function $f$ is analytic on $\mathbb{C}$ and satisfies all
the required properties. This concludes the proof.
\end{proof}

\section{Bilinear forms of Schur powers of matrices according to a
graph}\label{S5}

Theorem \ref{thmB} demonstrates that functions $f$ mapping
$\bp_{G_n}((0,\infty))$ into $\bp_{G_n}$ are not necessarily absolutely
monotonic, even if the family of graphs $\{G_n\}_{n \geq }$ has unbounded
maximal degree. In this section, we prove our third main result by
showing how a natural stronger hypothesis implies that $f$ is absolutely
monotonic. We begin with some notation.

\begin{definition}\label{def:kG}
Given $A \in \br_n$, denote by $Q_A$ the associated quadratic form
$Q_A(x) := x^T A x$, with kernel $\ker Q_A := \{\beta \in \R^n :
Q_A(\beta) = 0\}$.
Also define $A^{\circ 0}$ to be the matrix with entries $(A^{\circ
0})_{ij} := 1 - \delta_{a_{ij},0}$ (where $\delta$ denotes the Kronecker
delta function). For $k \geq 1$, define
\[ N_k(A) := \bigcap_{m=0}^{k-1} \ker (Q_{A^{\circ m}}) \cap \{\beta \in
\R^n : \beta^T A^{\circ k} \beta > 0\}. \]

\noindent When $k=0$, we define $N_0(A) := \{ \beta \in \R^n : \beta^T
A^{\circ 0} \beta > 0 \}$. 
\end{definition}

Notice that for a given nonzero matrix $A \in \br_n$ and any $k \geq 1$,
the set $N_k(A)$ is contained in $\ker Q_{A^{\circ 0}}$, and hence lives
in a hypersurface of dimension strictly smaller than $n$. Thus $N_k(A)$
has zero $n$-dimensional Lebesgue measure. 

Before proving Theorem \ref{thmC}, we recall the strategy of the proof of
Theorem \ref{Tvasu} provided by Vasudeva in \cite[Theorem 6]{vasudeva79}.
A fundamental ingredient in {\it loc.~cit.}~consists of constructing
vectors belonging to the kernel of bilinear forms associated to the Schur
powers of a matrix $A$. Using our notation, the first ingredient of the
proof in {\it loc.~cit.}~ is the following lemma.

\begin{lemma}\label{lem:k_G_complete}
For every $n \geq 2$, there exists a
positive semidefinite matrix $A$ such that $N_k(A) \not= \emptyset$ for
$k=1, \dots, n-1$. 
\end{lemma}

\begin{proof}
Let $\alpha_1, \dots, \alpha_n$ be $n$ distinct nonzero real numbers.
Define $\alpha^{(k)} := \left(\alpha_1^k, \dots, \alpha_n^k\right)^T$ for
$k \geq 0$, and $A := \alpha^{(1)} \alpha^{(1)T}$. Note that the vectors
$\alpha^{(0)}, \dots, \alpha^{(n-1)}$ are linearly independent, so given
$1 \leq k \leq n-1$, there exists $\beta_k \in \R^n$ which is orthogonal
to $\alpha^{(m)}$ for $m=0, \dots, k-1$, but not to $\alpha^{(k)}$. For
any $m \geq 0$, notice that $A^{\circ m} = \alpha^{(m)}\alpha^{(m)T}$.
Therefore $\beta_k \in \ker Q_{A^{\circ m}}$ for $m=1,\dots,k-1$, but
$\beta_k \not\in \ker Q_{A^{\circ k}}$. Finally, we have $\beta_k^T
A^{\circ k} \beta_k = (\beta_k^T \alpha^{(k)})^2 > 0$. Thus $\beta_k \in
N_k(A)$, showing that $k_G \geq n-1$.  
\end{proof}

The rest of the proof of Theorem \ref{Tvasu} goes as follows. Let $f: (0,
\infty) \rightarrow \mathbb{R}$ be such that $f[A] \in \bp_n$ for every
$A \in \bp_n((0, \infty))$. Consider the Taylor expansion of $f$ around
$a > 0$:
\[ f(a+t) = f(a) + f'(a) t + \dots + f^{(k-1)}(a) \frac{t^{k-1}}{(k-1)!}
+ f^{(k)}(a + \xi t) \frac{(\xi t)^k}{k!} \]

\noindent for some $0 < \xi < 1$. Denoting by $\one{n \times n}$ the $n
\times n$ matrix with every entry equal to $1$, we obtain: 
\[ f[a \one{n \times n} + tA] = f(a)\one{n \times n} + f'(a) tA + \dots +
f^{(k-1)}(a) \frac{t^{k-1}}{(k-1)!} A^{\circ(k-1)} + (f^{(k)}(a + t
\xi_{ij}))_{ij} \frac{t^k}{k!}\circ A^{k} \]

\noindent for some $0 < \xi_{ij} < 1$. Since $f[a \one{n \times n} + tA]
\in \bp_n$ by hypothesis, we obtain for any $\beta \in N_k(A)$:
\begin{equation}\label{eqn:taylor_f_A}
\beta^T f[a \one{n \times n} + tA] \beta = \beta^T \left( (f^{(k)}(a + t
\xi_{ij}))_{ij} \frac{t^k}{k!}\circ A^{k} \right) \beta \geq 0. 
\end{equation} 

\noindent Dividing by $t^k$ and letting $t \rightarrow 0^+$, it follows
that $f^{(k)}(a) \geq 0$. 

In light of Theorem \ref{thmA}, one can now ask if the above approach can
be adapted to the case of general graphs. A first difficulty arises when
trying to replace the matrix $\one{n \times n}$ in \eqref{eqn:taylor_f_A}
by $A_G+\Id_{|G|}$, where $A_G$ denotes the adjacency matrix of a graph
$G$. As shown by the following proposition, the matrix $A_G+\Id_{|G|} =
(A_G + \Id_{|G|})^{\circ 0}$ is never positive semidefinite if $G$ is not
a disconnected union of complete graphs.

\begin{proposition}\label{prop:adjmatrix}
Given $A \in \br_n$, the following are equivalent:
\begin{enumerate}
\item $A^{\circ 0}$ is positive semidefinite.
\item There exists a permutation matrix $P$ such that $P A^{\circ 0} P^T
= {\bf 0}_{n_0 \times n_0} \oplus \Id_{(n-n_0) \times (n-n_0)}$ for some
$0 \leq n_0 \leq n$.
\end{enumerate}
\end{proposition}

\noindent The proof is standard and resembles that of Proposition
\ref{thm:f(0)}, and is therefore omitted. See also \cite[Theorem
1.13]{Horn} for more equivalent conditions. 

A second major drawback in trying to adapt the proof of \cite[Theorem
6]{vasudeva79} is provided by the following result, which shows that for
large families of graphs $G$, the sets $N_k(A)$ can be empty for all
matrices in $\bp_G$. 

\begin{theorem}\label{thm:star_kG}
Let $G$ be a star graph with at least two vertices. Then $N_k(A)$ is
empty for all $k > 2$ and all positive semidefinite $A \in \bp_G$.
\end{theorem}

\begin{proof}
We will prove the following claim, which implies the assertion:
\begin{equation}\label{E3}
\ker Q_A \cap \ker Q_{A \circ A} = \bigcap_{m \geq 1} \ker Q_{A^{\circ
m}}.
\end{equation}

\noindent To show the claim, suppose $A \in \bp_G$ is as in the statement
of Proposition \ref{thm:charac_pd_star}, with $d \geq 1$. Then properties
(1)-(3) in that result hold here. Now define $a_m, L_m$ as in
\eqref{Estar}. Then since $p_i \geq 0$ for all $i$ and $a_1 \geq 0$,
hence
\[ \sum_{i > 1 \ : \ p_i \neq 0}
\frac{\alpha_i^{2m}}{p_i^m} \leq \left( \sum_{i > 1 \ : \ p_i \neq 0}
\frac{\alpha_i^2}{p_i} \right)^m \leq p_1^m. \]

\noindent This implies $a_m \geq 0 \ \forall m>0$, so $L_m$ is a real
matrix for all $m>0$. Moreover, $A^{\circ m} = L_m L_m^T$ for all $m >
0$, so $A^{\circ m}$ is also positive semidefinite. Now if $Q_{A^{\circ
m}}(\beta) = ||L_m^T \beta||^2 = 0$ for some $m>0$, then $ L_m^T \beta =
0$. Denoting $\beta = (\beta_1, \dots, \beta_{d+1})^T$, the condition
$L_m^T \beta = 0$ translates into the following equivalent conditions for
every $m>0$:
\begin{equation}\label{E1}
Q_{A^{\circ m}}(\beta) = 0 \quad \Leftrightarrow \quad L_m^T \beta = 0
\quad \Leftrightarrow \quad (\ \beta_1 a_m = 0, \quad \beta_1 \alpha_i^m
+ \beta_i p_i^m = 0\ \forall 2 \leq i \leq d+1 \ ).
\end{equation}

\noindent (Note that we use the characterization (2) in the statement of
Proposition \ref{thm:charac_pd_star}.) Now consider any vector $\beta \in
\ker Q_A \cap \ker Q_{A \circ A}$ such that $\beta_1 = 0$. Then by
Equation \eqref{E1} for $m=1$, either $\beta_i$ or $p_i$ is zero for all
$i>1$. But then $\beta_1 \alpha_i^m + \beta_i p_i^m = 0$ for all $m>0$
and all $i>1$. Moreover, $\beta_1 a_m = 0$ for all $m$. Hence by Equation
\eqref{E1}, $Q_{A^{\circ m}}(\beta) = 0$ for all $m>0$, as desired.

Next, assume that $\beta \in \ker Q_A \cap \ker Q_{A \circ A}$ and
$\beta_1 \neq 0$. Then Equation \eqref{E1} holds for $m=1,2$. We now
claim that all $2 \leq i \leq d+1$ fall into exactly one of the following
three categories:
\begin{itemize}
\item Suppose $p_i = 0$ for some $2 \leq i \leq d+1$. Then $\alpha_i = 0$
by Proposition \ref{thm:charac_pd_star}, so $\beta_1 \alpha_i^m + \beta_i
p_i^m = 0\ \forall m>0$.

\item Suppose $p_i \neq 0$ but $\alpha_i = 0$. Then $\beta_i = 0$ by
Equation \eqref{E1} for $m=1$, so once again, $\beta_1 \alpha_i^m +
\beta_i p_i^m = 0$ for all $m>0$.

\item Suppose $p_i, \alpha_i \neq 0$. Then by Equation \eqref{E1} for
$m=1,2$, $\beta_i = - \beta_1 \alpha_i/p_i = - \beta_1 \alpha_i^2/p_i^2$.
This implies that $\alpha_i = p_i \neq 0$, whence $\beta_i = - \beta_1$.
Once again, this implies that $\beta_1 \alpha_i^m + \beta_i p_i^m = 0$
for all $m>0$.
\end{itemize}

\noindent Thus we see that the second part of the last equivalent
assertion in Equation \eqref{E1} holds in all three cases above, for all
$m>0$. It remains to prove that $a_m = 0$ for all $m>0$ (since $\beta_1
\neq 0$). Now define $c_i := 0$ if $p_i = 0$, and $\alpha_i^2 / p_i$
otherwise. Then $a_m = p_1^m - \sum_{i=2}^{d+1} c_i^m$, so using $a_1 = 0
= a_2$ from Equation \eqref{E1} implies:
\[ \sum_{i=2}^{d+1} c_i = p_1, \qquad \sum_{i=2}^{d+1} c_i^2 = p_1^2 =
\left( \sum_{i=2}^{d+1} c_i \right)^2. \]

\noindent Since $c_i \geq 0$ for all $i$, this system of equations has no
solutions if even two $c_i$ are positive. We thus conclude that $c_i > 0$
for at most one $i$, say $c_i = 0$ if $i \neq i_0$. Then $p_1 = c_{i_0}$
(since $a_1 = 0$). Hence,
\[ a_m = p_1^m - \sum_{i=2}^{d+1} c_i^m = c_{i_0}^m - c_{i_0}^m - 0 = 0
\qquad \forall m>0, \]

\noindent as desired. This proves the claim.
\end{proof}


\subsection{Proof of Theorem \ref{thmC}}

Proposition \ref{prop:adjmatrix} and Theorem \ref{thm:star_kG}
demonstrate that one faces major obstacles when trying to generalize the
argument in \cite[Theorem 6]{vasudeva79} to arbitrary graphs $G$. New
tools are required. In the rest of the paper, we carefully study bilinear
forms associated to the Schur powers of matrices in $\bp_G$ for an
arbitrary graph $G$, and use this analysis to prove Theorem \ref{thmC}.
First, we introduce some notation.

\begin{definition}
Given a graph $G$, let 
\begin{equation}
k_G := \max_{A \in \br_G} \max \left\{ k \geq 0 : N_m(A) \not= \emptyset
\textrm{ for } m = 1, \dots, k \right\}.
\end{equation}
\end{definition}

Note also that for any pair of
graphs $G$ and $H$, 
\begin{equation}\label{eqn:k_G_k_H}
H \subseteq G \quad \implies \quad k_H \leq k_G.
\end{equation}

Theorem \ref{thm:bound_k_G_degree} below provides bounds for the
constants $k_G$, and will be crucially used in the proof of Theorem
\ref{thmC} as a replacement of Lemma \ref{lem:k_G_complete} for a general
graph $G$. Recall that $\Delta(G)$ denotes the maximum vertex degree of
the graph $G$.

\begin{theorem}\label{thm:bound_k_G_degree}
For all graphs $G$ with at least one edge, we have 
\begin{equation}
\max(2,\Delta(G)) \leq k_G < |V(G)| + |E(G)| = \dim_\mathbb{R} \br_G.
\end{equation} 
\end{theorem}

\begin{proof}
We begin by proving the upper bound. Given a symmetric matrix $A$, denote
by $\eta(A)$ is the number of distinct nonzero entries of $A$. We claim
that for any symmetric $A$, 
\begin{equation}\label{eqn:ker_eq}
\bigcap_{k=0}^{\eta(A)-1} \ker Q_{A^{\circ k}} = \bigcap_{k \geq 0}
\ker Q_{A^{\circ k}}.
\end{equation}

\noindent In particular, $N_k(A) = \emptyset\ \forall k \geq \eta(A)$.
The $\subseteq$ inclusion in Equation \eqref{eqn:ker_eq} is obvious. To
prove the reverse inclusion, let $A \in \br_n$, define $d := \eta(A)$,
and let $\{ \alpha_1, \dots, \alpha_d \}$ be the distinct nonzero entries
of $A$. Given a vector $\beta = (\beta_1, \dots, \beta_n)$, and $1 \leq l
\leq d$, define:
\begin{equation}
S_l := \{ (i,j) : 1 \leq i,j \leq n, a_{ij} = \alpha_l \}, \qquad
\beta'_l := \sum_{(i,j) \in S_l} \beta_i \beta_j.
\end{equation}

\noindent Then $Q_{A^{\circ k}}(\beta) = \sum_{l=1}^d \alpha_l^k
\beta'_l$. Let $B$ be the $d \times d$ Vandermonde matrix whose
$(i,j)$th entry is $\alpha_i^{j-1}$ for $1 \leq i,j \leq d$; then $B$
is non-singular. Also define $\beta' := (\beta'_1, \dots, \beta'_d)^T$.
Then $\beta \in \ker Q_{A^{\circ k}}$ for all $0 \leq k < d$ if and only
if $B \beta' = 0$, if and only if $\beta'_l = 0$ for all $l$. But then
$\beta \in \cap_{k \geq 0} \ker Q_{A^{\circ k}}$. This proves the reverse
inclusion, and hence Equation \eqref{eqn:ker_eq}. To conclude the proof
of the upper bound, note that if $k \geq \eta(A)$, then by Equation
\eqref{eqn:ker_eq},
\[ N_k(A) =\cap_{m=0}^{k-1} \ker (Q_{A^{\circ m}}) \cap \{\beta \in \R^n
: \beta^T A^{\circ k} \beta > 0\} = \emptyset. \]

\noindent Therefore $k_G < \max_{A \in \br_G} \eta(A) = |V(G)| + |E(G)|$.

We now prove the lower bound for $k_G$. To show that $k_G \geq 2$, it
suffices by Equation \eqref{eqn:k_G_k_H} to show that $k_{K_2} \geq 2$.
Hence, suppose $V(G) = \{ 1, 2 \}$ and $E = \{ (1,2) \}$. Now fix
positive numbers $a \neq b > 0$ and consider the matrix $A_j =
\begin{pmatrix} a & \frac{a+b}{2(3-j)}\\ \frac{a+b}{2(3-j)} &
b\end{pmatrix}$ for $j=1,2$. It is clear that $\beta := (1,-1)^T$ is in
$N_j(A_j)$ for $j=1,2$. Hence $k_G \geq k_{K_2} \geq 2$.

Finally, we show that $k_G \geq \Delta(G)$. For ease of exposition, we
divide this part of the proof into four steps.\medskip

\noindent {\bf Step 1:} We begin by introducing the key matrix $A$.
Without loss of generality, assume that $\deg_G 1 = \Delta(G) =: d>0$,
and $\{ 2, 3, \dots, d+1 \}$ are incident to $1$. Let $\alpha_i$ be
distinct nonzero real numbers for $1 \leq i \leq d+1$, and define
\begin{align}\label{eqn:star_matrix_A}
\alpha^{(k)} := &\ \left( \frac{\alpha_1^k}{2}, \alpha_2^k, \dots,
\alpha_{d+1}^k, 0, \dots, 0 \right)^T \in \R^{|G|} \quad \forall k \geq
0, \notag\\
A := &\ e_1 (\alpha^{(1)})^T + \alpha^{(1)} e_1^T \in \br_G ([0,\infty)).
\end{align}

\noindent It follows from Equation \eqref{eqn:ker_eq} that $N_k(A)$ is
empty if $k \geq d+1$. We will show that this bound is sharp for generic
$\alpha_i$. More precisely, we show in the remainder of the proof that
$N_1(A), \dots, N_d(A)$ are nonempty if the $\alpha_i$ are all nonzero
and distinct and less than $\alpha_1$ for $i>1$.

Since all $\alpha_i \neq 0$, the graph of $A^{\circ k}$ is a star graph
over $d+1$ vertices for all $k \geq 0$, and satisfies $A^{\circ k} \in
\br_G$. Moreover, $A^{\circ k}$ and $Q_{A^{\circ k}}(\beta)$ can be
easily computed for all $k \geq 0$:
\begin{equation}\label{Eqak}
A^{\circ k} = e_1 (\alpha^{(k)})^T + \alpha^{(k)} e_1^T \quad \implies
\quad Q_{A^{\circ k}}(\beta) = 2 (\beta^T e_1) (\beta^T \alpha^{(k)}) = 2
\beta_1 (\beta^T \alpha^{(k)}).
\end{equation}

\noindent Thus, $\ker Q_{A^{\circ k}} = \{ e_1 \}^\perp \cup \{
\alpha^{(k)} \}^\perp$.\medskip

\noindent {\bf Step 2:} Now define $V_k$ to be the span of $\alpha^{(0)},
\dots, \alpha^{(k-1)}$. We then claim that the following two assertions
hold:
\begin{align}
&\dim V_k = k,\ \forall \ 0 \leq k \leq d+1 \\
&e_1 \in V_k \quad \Longleftrightarrow \quad k = d+1.
\end{align}

\noindent (In other words, the vectors $\alpha^{(0)}, \dots,
\alpha^{(d)}$ are linearly independent - and this continues to hold if we
replace $\alpha^{(d)}$ by $e_1$.) The first assertion is immediate from
Vandermonde determinant theory.
To show the second assertion, consider the matrix whose columns are (the
first $d+1$ coordinates of) $e_1, \alpha^{(0)}, \dots, \alpha^{(d-1)}$.
Its determinant is equal to the minor obtained by deleting its first row
and first column. This minor is exactly the determinant of the
Vandermonde matrix whose columns are $\{ (\alpha_2^k, \dots,
\alpha_{d+1}^k)^T : 0 \leq k \leq d-1 \}$. Hence it is nonzero, whence
the second assertion follows.\medskip

\noindent {\bf Step 3:} The next step is to produce $\beta_k \in N_k(A)$
for $k=1, \dots, d-1$. We work in $V_{d+1}$ for the rest of this proof.
Define $\bp_{V_k^\perp}$ to be the projection operator onto $V_k^\perp$.
Now given $0<k<d$, note that $\alpha^{(0)}, \dots, \alpha^{(k)}, e_1$ are
linearly independent from above. Therefore
$\bp_{V_k^\perp}(\alpha^{(k)})$ and $\bp_{V_k^\perp}(e_1)$ are also
linearly independent,
so in particular, they have an angle of less than $180^\circ$ between
them. Choose $\beta_k$ to be a positive scalar multiple of the unique
angle bisector in the plane spanned by $\bp_{V_k^\perp}(\alpha^{(k)})$
and $\bp_{V_k^\perp}(e_1)$. More precisely, set $\beta_k := a_k + b_k$,
where
\[ a_k := \frac{\bp_{V_k^\perp}(\alpha^{(k)})}{||
\bp_{V_k^\perp}(\alpha^{(k)})||}, \qquad b_k :=
\frac{\bp_{V_k^\perp}(e_1)}{||\bp_{V_k^\perp}(e_1)||}. \]

\noindent Let $\gamma := ||\bp_{V_k^\perp}(\alpha^{(k)})|| \cdot
||\bp_{V_k^\perp}(e_1)|| > 0$. Since $a_k, b_k$ are unit vectors, it is
clear that
\begin{align*}
Q_{A^{\circ k}}(\beta_k) = &\ 2 (\beta_k^T e_1) (\beta_k^T \alpha^{(k)})
= 2 (\beta_k^T \bp_{V_k^\perp}(e_1)) (\beta_k^T
\bp_{V_k^\perp}(\alpha^{(k)}))\\
= &\ 2 \gamma ((a_k+b_k)^T b_k) ((a_k+b_k)^T a_k) = 2 \gamma (1 +
(a_k,b_k))^2 > 0.
\end{align*}

\noindent The last inequality here is strict by the Cauchy-Schwartz
inequality, since $a_k,b_k$ are not proportional from above. On the other
hand, if $0 \leq m < k$, then $Q_{A^{\circ m}}(\beta_k) = 2 (\beta_k^T
e_1) (\beta_k^T \alpha^{(m)}) = 0$, since $\alpha^{(m)} \in V_k$ and
$\beta_k \in V_k^\perp$. We conclude that $\beta_k \in N_k(A)$ for
$0<k<d$. (In particular, using any set of distinct nonzero $\alpha_i$, we
obtain that $k_G \geq \Delta(G) - 1$.)\medskip

\noindent {\bf Step 4:} Finally, we produce $\beta_d \in N_d(A)$. To do
so, note that $V_d$ is a codimension one subspace in $V_{d+1}$, so $\dim
V_d^\perp = 1$. Note that this uniquely determines $\beta_d \in
V_d^\perp$ up to multiplying by a nonzero scalar $c \neq 0$; moreover,
the sign of $Q_{A^{\circ d}}(\beta_d)$ is independent of $c$.

Note that $e_1, \alpha^{(d)} \notin V_d$. Hence \eqref{Eqak} implies the
following: if $\pm \beta_d$ are the only two unit vectors in $V_d^\perp$,
then $Q_{A^{\circ d}}(\beta_d) = 2 (\beta_d^T e_1) (\beta_d^T
\alpha^{(d)})$.
It is clear from the above remarks that $Q_{A^{\circ d}}(\beta_d)$ is
positive if and only if $e_1$ and $\alpha^{(d)}$ are on the ``same side"
of the hyperplane $V_d$ in $V_{d+1}$ (i.e., their inner products with
$\beta_d$ have the same sign). In order to ensure this, one now needs a
constraint on the $\alpha_i$. Thus, consider a matrix $A_d(X)$, whose
columns are $\alpha^{(0)}, \dots, \alpha^{(d-1)}, x$, where $x := (x_1,
\dots, x_{d+1})^T$ is a column vector of variables. It is clear that
$\det A_d(x) = \sum_{i=1}^{d+1} c_i x_i$, for some scalars $c_i$.
Moreover, $\det A_d(v) = 0$ if we replace $x$ by any vector $v \in V_d$,
since the other columns form a basis of $V_d$.

In order that the two vectors $e_1, \alpha^{(d)} \in V_{d+1}$ lie on the
same side of $V_d$ (in $V_{d+1}$), it is enough to ensure that $\det
A_d(e_1), \det A_d(\alpha^{(d)})$ have the same sign, i.e., $\det
A_d(e_1) \cdot \det A_d(\alpha^{(d)}) > 0$. Now from the above remarks
and the standard Vandermonde formula, we compute:
\[ \det A_d(\alpha^{(d)}) = \prod_{1 \leq i < j \leq d+1} (\alpha_j -
\alpha_i), \qquad \det A_d(e_1) = (-1)^d \prod_{2 \leq i < j \leq d+1}
(\alpha_j - \alpha_i). \]

\noindent Since all $\alpha_i$ are pairwise distinct, upon removing the
perfect squares we obtain the condition needed to ensure that $e_1$ and
$\alpha^{(d)}$ are on the same side of $V_d$; namely,
\[ (-1)^d \prod_{j=2}^{d+1} (\alpha_j - \alpha_1) = \prod_{j=2}^{d+1}
(\alpha_1 - \alpha_j) > 0. \]

\noindent This inequality holds if we choose $\alpha_1 > \max(\alpha_2,
\dots, \alpha_{d+1})$. Thus we have produced a matrix $A = e_1
(\alpha^{(1)})^T + \alpha^{(1)} e_1^T \in \br_G$ and a vector $\beta_d
\in N_d(A)$, in addition to the vectors $\beta_k \in N_k(A)$ for $0 < k <
d$ (constructed above for any nonzero distinct $\alpha_i$). This
concludes the proof.
\end{proof}

\begin{remark}
Note that the bounds in Theorem \ref{thm:bound_k_G_degree} are sharp for
$G = K_2$.
\end{remark}

We now proceed to prove Theorem \ref{thmC} using Theorem
\ref{thm:bound_k_G_degree}.

\begin{proof}[Proof of Theorem \ref{thmC}]
Suppose $f$ is any function satisfying the hypotheses of the theorem. By
considering the matrix $M = a {\bf 1}_{1 \times 1} \oplus {\bf
0}_{(|G_1|-1) \times (|G_1|-1)}$ for $a \in I$, it follows from the
hypotheses that $f(I) \subset [0,\infty)$. We next prove that $f$ is
continuous on $I$. By Remark \ref{rem:vasudevaM2} and Theorem
\ref{thm:vasudeva:M2}, $f$ is necessarily continuous and increasing on
$(0,R)$, and so $f^+(0) := \lim_{x \rightarrow 0^+} f(x)$ exists. Since
$\Delta(G_n) \to \infty$, choose $n$ such that $G_n$ contains $K_3$ or
$A_3$ as an induced subgraph. Without loss of generality, assume that
vertex $1$ is connected to vertices $2,3$. To prove that $f$ is
continuous at $0$, first note that $tB(2,1,1) \oplus {\bf 0}_{(n-3)
\times (n-3)} \in \bp_G(I)$ for $t>0$ small enough, where $B(2,1,1)$ was
defined in Equation \eqref{eqn:B_3}. Therefore $f[tB(2,1,1)] \in \bp_3$.
Since $f$ is absolutely monotonic on $(0,R)$, it is nonnegative and
increasing there, and so $f^+(0) := \lim_{x \rightarrow 0^+} f(x)$
exists. As a consequence,
\begin{equation}\label{eqn:f_plus}
\lim_{t \rightarrow 0^+} f[tB(2,1,1)] =
\begin{pmatrix}
f^+(0) & f^+(0) & f^+(0)\\
f^+(0) & f^+(0) & f(0)  \\
f^+(0) & f(0) & f^+(0)
\end{pmatrix} \in \bp_3. 
\end{equation}

\noindent Computing the determinant of the above matrix, we conclude that
$f^+(0) = f(0)$, i.e., $f$ is continuous at $0$. Therefore the function
$f$ is now continuous on $I$. 

Next suppose that $f \in C^\infty(I)$, and fix $k > 0$ and $a \in (0,R)$.
We claim that $f^{(k)}(a) \geq 0$. To show the claim, choose $n \in \N$
such that $\Delta(G_n) \geq k$. By Theorem \ref{thm:bound_k_G_degree},
there exists $A \in \br_{G_n}$ and $\beta \in \R^n$ such that $\beta \in
N_k(A)$. By the definition of $N_k(A)$, we have $\beta^T (a A^{\circ 0} +
tA) \beta \geq 0$ for every $t > 0$ and thus, by hypothesis,
\[ \beta^T f_{G_n}[A^{\circ 0} + tA] \beta \geq 0, \quad \forall 0 < t <
\epsilon, \quad \mbox{where} \quad \epsilon :=
\min\left(\frac{a}{\max_{i,j} |a_{ij}|}, \frac{R-a}{\max_{i,j}
|a_{ij}|}\right). \]

\noindent Note that $(a-\epsilon,a+\epsilon) \subset (0,R) \subset {\rm
dom}(f)$. Now expanding $f$ in Taylor series around $A$, we obtain:
\[ f_{G_n}[A^{\circ 0} + tA] = \sum_{r=0}^{k-1} \frac{f^{(r)}(a)}{r!}
(tA)^{\circ r} + \left(f^{(k)}(a + \theta_{ij}  t a_{ij})\right)_{ij}
\circ \frac{t^{k}}{k!} A^{\circ k} \]

\noindent where $0 < \theta_{ij} < 1$. In particular, 
\[ \beta^T f_{G_n}[A^{\circ 0} + tA] \beta = \sum_{i,j=1}^{n} \beta_i
\beta_j f^{(k)}(a + \theta_{ij} t a_{ij}) \frac{t^{k}}{k!} a_{ij}^{k}
\geq 0, \]

\noindent since $\beta \in N_k(A)$.
Now divide by $t^{k} / k!$ and let $t \rightarrow 0^+$ to obtain:
$f^{(k)}(a)  \left(\beta^T A^{\circ k} \beta\right) \geq 0$.
The claim follows since $\beta^T A^{\circ k} \beta > 0$ by hypothesis.
Theorem \ref{thm:abs_monotonic_equiv} now implies that $f$ is analytic
and absolutely monotonic on $I$. This shows the result when $f \in
C^\infty(I)$.

Finally, suppose $f$ is continuous but not necessarily smooth on $I$, and
let $0 < b < R$. For any probability distribution $\phi \in C^\infty(\R)$
with compact support in $(b/R, \infty)$, let 
\[ f_\phi(x) := \int_{b/R}^\infty f(xy^{-1}) \phi(y) \frac{dy}{y}, \qquad
0 < x < b. \]

\noindent Then $f_\phi \in C^\infty(0,b)$. Suppose $\beta^T M \beta \geq
0$ for some $M \in \br_{G_n}((0,b))$ and some $n \geq 1$. Then,
\[ \beta^T (f_\phi)_{G_n}[M] \beta
= \int_{b/R}^\infty \sum_{i,j=1}^{|G_n|}  \beta_i \beta_j f(m_{ij}y^{-1})
\phi(y)\frac{dy}{y} = \int_{b/R}^\infty \beta^T f_{G_n}[y^{-1}M] \beta
\phi(y)\frac{dy}{y}. \]

\noindent Notice that the integrand is non-negative for every $y > 0$. It
thus follows that $\beta^T (f_\phi)_{G_n}[M] \beta \geq 0$.

Now consider a sequence $\phi_m \in C^\infty(\R)$ of probability
distributions with compact support in $(b/R,\infty)$ such that $\phi_m$
converges weakly to $\delta_1$, the Dirac measure at $1$. Note that such
a sequence can be constructed since $b/R < 1$. By the above analysis in
this proof, $f_{\phi_m}$ is absolutely monotonic on $(0,b)$ for every $m
\geq 1$.
Therefore by Theorem \ref{thm:abs_monotonic_equiv}, the forward
differences $\Delta^k_h[f_{\phi_m}](x)$ of $f_{\phi_m}$ are nonnegative
for $l \geq 0$ and all $x$ and $h$ such that $0 \leq x < x+h < \dots <
x+lh < R$. 
Since $f$ is continuous, $f_{\phi_m}(x) \rightarrow f(x)$ for every $x
\in (0,b)$. Therefore $\Delta^k_h[f](x) \geq 0$ for all such $x$ and $h$
as well. As a consequence, by Theorem \ref{thm:abs_monotonic_equiv}, the
function $f$ is absolutely monotonic on $(0,b)$. Since this is true for
every $0 < b < R$, it follows that $f$ is absolutely monotonic on $I$.
\end{proof}

\begin{remark}
In Theorem \ref{thmC}, the assumptions only have to be verified for an
appropriate sequence of matrices $(M_n)_{n \geq 1}$ such that $M_n \in
\br_{G_n}(I)$, and a sequence of vectors $\beta_{n,k}$ such that
$\beta_{n,k} \in N_k(M_n)$ for $1 \leq k \leq \Delta(G_n)$. Moreover, it
also suffices to verify the hypotheses of Theorem \ref{thmC} for matrices
of the form $aA^{\circ 0} + t A$ for $a, t > 0$ and where $A$ has the
form \eqref{eqn:star_matrix_A}.
\end{remark}



\end{document}